\documentclass[review, compress, 3p]{elsarticle}
\usepackage{lineno,hyperref}
\usepackage{amssymb,amsfonts}
\usepackage{amsmath,amsthm}
\usepackage{algorithm,algorithmic}
\usepackage{graphicx,subfigure}
\usepackage{float,rotating}
\usepackage{booktabs}
\usepackage{multirow,multicol}
\usepackage{lscape}
\usepackage{epstopdf}
\usepackage{color}
\usepackage{setspace}
\usepackage{natbib}
\usepackage{bm}
\usepackage{url,doi}

\makeatletter
\@addtoreset{equation}{section}
\makeatother

\newtheorem{theorem}{Theorem}[section]
\newtheorem{lemma}{Lemma}[section]
\newtheorem*{definition}{Definition}
\newtheorem{remark}{Remark}
\newtheorem{corollary}{Corollary}

\journal{}

\begin{document}

\begin{frontmatter}

\title{An efficient second-order energy stable BDF scheme for the space fractional Cahn-Hilliard equation}

\author[address1,address2]{Yong-Liang Zhao}
\ead{ylzhaofde@sina.com}

\author[address3]{Meng Li\corref{correspondingauthor}}
\cortext[correspondingauthor]{Corresponding author}
\ead{limeng@zzu.edu.cn}

\author[address2]{Alexander Ostermann}
\ead{alexander.ostermann@uibk.ac.at}

\author[address4]{Xian-Ming Gu}
\ead{guxianming@live.cn}

\address[address1] {School of Mathematical Sciences, \\
University of Electronic Science and Technology of China, \\
Chengdu, Sichuan 611731, P.R. China}
\address[address2]{Department of Mathematics, University of Innsbruck, Technikerstra{\ss}e 13, Innsbruck 6020, Austria}
\address[address3] {School of Mathematics and Statistics, Zhengzhou University,
Zhengzhou, Henan 450001, P.R. China}
\address[address4] {School of Economic Mathematics/Institute of Mathematics, \\
Southwestern University of Finance and Economics, Chengdu, Sichuan 611130, P.R. China}

\begin{abstract}
The space fractional Cahn-Hilliard phase-field model is more adequate and accurate in the description of the
formation and phase change mechanism than the classical Cahn-Hilliard model.
In this article, we propose a temporal second-order energy stable scheme for the space fractional Cahn-Hilliard model.
The scheme is based on the second-order backward differentiation formula in time and a finite difference method in space.
Energy stability and convergence of the scheme are analyzed,
and the optimal convergence orders in time and space are illustrated numerically.
Note that the coefficient matrix of the scheme is a $2 \times 2$ block matrix with a Toeplitz structure in each block.
Combining the advantages of this special structure with a Krylov subspace method,
a preconditioning technique is designed to solve the system efficiently.
Numerical examples are reported to illustrate the performance of the preconditioned iteration.
\end{abstract}

\begin{keyword}
Space fractional Cahn-Hilliard equation\sep Energy stability\sep Convergence\sep Newton's method\sep
Krylov subspace method\sep Preconditioning
\MSC[2010] 65M06\sep 65M12\sep 65N06
\end{keyword}

\end{frontmatter}

%\linenumbers

\section{Introduction}
\label{sec1}

The classical Cahn-Hilliard (CH) equation originally proposed by Cahn and Hilliard \cite{cahn1958free}:
\begin{equation*}
\partial_t \phi = M \Delta \mu, \qquad
\mu = -\varepsilon^2 \Delta \phi + F'(\phi)~
\left( \textrm{where}~F(\phi) = \frac{1}{4} \phi^4 - \frac{1}{2} \phi^2 \right)
\end{equation*}
%with $F(\phi) = \frac{1}{4} \phi^4 - \frac{1}{2} \phi^2$,
is well-known in modeling spinodal decomposition in a binary alloy, see \cite{bertozzi2006inpainting,cohen1981generalized,klapper2006role,farshbaf2015phase} for other applications.
Its free energy functional has the form
$%\begin{equation*}
\int_{\Omega} \left( \frac{\varepsilon^2}{2} \left| \nabla \phi \right| ^2 + F(\phi) \right) d \textbf{x}
$, %\end{equation*}
where $\Omega \subset \mathbb{R}^d~(d \in \mathbb{N}^{+})$.
In phase-field models with the double-well potential $F(\phi)$,
$\mu$ is treated as the indicator of the concentration (or volume fraction) of one fluid at the location $\textbf{x}$
in the immiscible mixture with the second fluid \cite{wang2019finite}.
The key feature of the CH equation is that the continuous scalar field $\phi$ describes surfaces and interfaces implicitly.
It takes constant values in the bulk phases but varies rapidly across its diffuse fronts.
When solving the CH equation numerically, finite difference
and finite element methods are the most common discretization methods,
see the related Refs.~\cite{barrett1999finite,elliott1989second,wise2007solving,feng2016analysis,
yan2018second,cheng2019energy}
and the references therein.

In the original formulation of the physical model \cite{cahn1958free}, the term $\Delta \phi$ in the CH equation,
which describes long-range interactions among particles, should be replaced by a spatial convolution or a nonlocal integral term.
However, in the subsequent mathematical literature, such a nonlocal term has been substituted with $\Delta \phi$ mainly for analytical reasons.
Some important information coming from long-range interactions may be lost through this replacement.
Under this perspective, using the fractional Laplacian operator $(-\Delta)^{\alpha/2}$
appears to be more adherent to the physical setting.
Along this line, several fractional versions of the CH equation have been proposed in \cite{wang2019finite,abels2015cahn,akagi2016fractional,ainsworth2017analysis,ainsworth2017well,weng2017fourier}.
Ainsworth and Mao \cite{ainsworth2017analysis} proposed the Fourier-Galerkin scheme
to approximate the space fractional CH equation.
Weng et al.~\cite{weng2017fourier} developed an unconditionally energy stable Fourier spectral scheme to solve
the space fractional CH phase-field model \cite{akagi2016fractional} numerically.
Zhai et al.~\cite{zhai2019numerical} derived a fast Strang splitting method to solve the fractional CH equation.
Moreover, Liu et al.~\cite{liu2018time} derived an efficient Fourier spectral scheme
to approximate the time fractional CH equation.
Tang et al.~\cite{tang2019energy} proposed a class of  finite difference schemes, which can inherit the energy stability,
to solve the time-fractional phase-field models, e.g., the time fractional CH equation.

In this paper, we design an efficient finite difference scheme
for numerically solving the space fractional CH (SFCH) model equipped with homogeneous boundary conditions \cite{wang2019finite, akagi2016fractional}:
\begin{equation}
\begin{cases}
\partial_t \phi(x,t) = -(-\Delta)^{\alpha/2} \mu(x,t), & \textrm{in} \quad \Omega \times (0,T], \\
\mu(x,t) = \phi^3(x,t) - \phi(x,t) + \varepsilon^2 (-\Delta)^{\alpha/2} \phi(x,t),
& \textrm{in} \quad \Omega \times (0,T], \\
\phi(x,t) = \mu(x,t) = 0, & \textrm{in} \quad \mathbb{R} \setminus \Omega \times (0,T],\\
\phi(x,0) = \phi_0(x), &  \textrm{in} \quad \Omega.
\end{cases}
\label{eq1.1}
\end{equation}
Here $1 < \alpha < 2$, $\Omega = (-L,L) \subset \mathbb{R}$ and $\varepsilon$ (a positive constant) is a length scale parameter.
Note that the model \eqref{eq1.1} is a special case of \cite{wang2019finite, akagi2016fractional}.
For $\alpha = 2$, Eq.~\eqref{eq1.1} becomes the classical CH equation.
It is worth emphasising that in this work, we consider the exterior Dirichlet boundary condition:
$\phi(x,t) = \mu(x,t) = 0~\mathrm{in}~\mathbb{R} \setminus \Omega \times (0,T]$
instead of the usual one $\phi(x,t) = \mu(x,t) = 0~\textrm{on}~\partial \Omega \times (0,T]$.
A reason for considering such a boundary condition is that in a microscopic view, a particle may not stop on the boundary
and may directly jump to the outside of the system.
%Moreover, the first two equations in \eqref{eq1.1} cannot be combined into a single one (at least in a strong formulation),
%i.e., $\partial_t \phi(x,t) = -(-\Delta)^{\alpha/2} \left(
%\phi^3(x,t) - \phi(x,t) + \varepsilon^2 (-\Delta)^{\alpha/2} \phi(x,t) \right)$,
%since the fractional Laplace operator is nonlocal and it also depends on the unknown values outside of $\Omega$.

Energy stability plays an essential role in the accuracy
of long-time numerical simulations of the (fractional) CH equation.
Wang et al.~\cite{wang2019finite} proposed a linear finite element algorithm for the SFCH model \eqref{eq1.1},
which, however, is only first-order accurate in time.
Designing a temporal high order scheme for such a model might be difficult.
%because the more complicated form than the first-order one for the nonlinear terms can be obtained.
In this work, we derive and analyze an alternative second-order energy stable scheme for the SFCH model \eqref{eq1.1}.
The new scheme is based on the second-order backward differentiation formula (BDF2).
In this scheme, the nonlinear and surface diffusion terms are updated implicitly,
whereas the concave diffusion term is approximated by a second-order accurate explicit extrapolation formula.
However, this explicit extrapolation cannot assure energy stability.
To save the energy stability of the numerical scheme,
we add a second-order Douglas-Dupont-type regularization term, i.e., $\sigma \tau (-\Delta)^{\alpha/2} (\phi^{j + 1} - \phi^{j})$,
where $\tau$ is the temporal step size.
A careful analysis suggests that the energy stability is guaranteed under the mild condition $\sigma \geq 1/16$.

It is easy to see that our scheme results in a nonlinear time-stepping system
with a $2 \times 2$ block matrix.
Because of the nonlocal property of the fractional Laplace operator,
traditional methods for linear problems (e.g., Gaussian elimination) for solving such a system
need $\mathcal{O}(N^3)$ operations per iteration step and the storage requirement is of $\mathcal{O}(N^2)$,
where $N$ is the number of space grid points.
However, we can use that each block of the matrix has a Toeplitz structure.
Thanks to this structure, matrix-vector multiplications
can be computed in $\mathcal{O}(N \log N)$ operations via fast Fourier transforms (FFTs)
\cite{ng2004iterative, chan2007introduction}.
With this advantage, in each iterative step of a Krylov subspace method, the memory requirement and computational cost
are $\mathcal{O}(N)$ and $\mathcal{O}(N \log N)$, respectively.
Since the ill-conditioning of the nonlinear system
causes a slow convergence of the Krylov subspace method,
a preconditioning technique is designed to solve the problem efficiently.
For many other studies about Toeplitz-like systems,
see \cite{gu2017fast, li2018fast, lei2013circulant, zhao2018limited, zhao2020tempered, li2018fastamc} and the references therein.

The main contributions of this work can be summarized as follows:

(i) An energy stable scheme with second-order accuracy in time is proposed,
and its stability and convergence are analyzed.

(ii) Based on the special structure of the coefficient matrix,
we introduce a block lower triangular preconditioner $P_{L}$, whose $(2,2)$ block is a circulant matrix.
This accelerates the solution of the arising nonlinear system.
If the circulant matrix is replaced by the Schur complement (denoted as $S$),
the performance of $P_{L}$ will become better.
However, $S$ is a general dense matrix in our case, and the computation cost of $S^{-1}$ is very high.
To overcome this drawback, the Schur complement is replaced by a new Strang-type skew-circulant matrix,
see Section \ref{sec4} for details.
Numerical experiences show that this Strang-type skew-circulant matrix
performs slightly better than Strang's circulant preconditioner \cite{ng2004iterative, chan2007introduction}.

The rest of this paper is organized as follows. Some auxiliary notations and the energy function are introduced in Section \ref{sec2}.
Section \ref{sec3} derives our energy stable scheme and provides the stability and convergence analysis.
In Section \ref{sec4}, the block lower triangular preconditioner and the new Strang-type skew-circulant matrix
are designed.
Several numerical examples are provided in Section \ref{sec5} to verify the convergence order of our scheme
and to show the performance of the preconditioning technique. Some conclusions are drawn in Section \ref{sec6}.

\section{Preliminaries}
\label{sec2}

In this section the definition of the fractional Laplace operator $(-\Delta)^{\alpha/2}$ and the energy of \eqref{eq1.1} are given.
The frequently used definitions of $(-\Delta)^{\alpha/2}$ are based on Fourier representation or on singular integral representations,
see \cite{kwasnicki2017ten} for other equivalent definitions.
In this paper, different from the previous works \cite{ainsworth2017analysis,ainsworth2017well,weng2017fourier},
the following singular integral \cite{landkof1972foundations,samko1993fractional} is considered
for the definition of the one-dimensional fractional Laplacian.
% Definitation
\begin{definition}(\cite{landkof1972foundations,samko1993fractional})
For $1 < \alpha < 2$ and $\phi(x)$ in the Schwartz class $\mathcal{S}(\mathbb{R})$ of the rapidly decaying functions at infinity,
the fractional Laplacian of order $\frac{\alpha}{2}$ is defined as follows
\begin{equation}
(-\Delta)^{\alpha/2} \phi(x) =
\begin{cases}c_{1,\alpha} \textrm{P.V.}
\int_{\mathbb{R}} \frac{\phi(x) - \phi(y)}{\left| x - y \right|^{1 + \alpha}} dy, & x \in \Omega, \\
0, & x \in \mathbb{R} \setminus \Omega,
\end{cases}
\label{eq2.1}
\end{equation}
where
$c_{1,\alpha} = \frac{2^{\alpha - 1} \alpha \Gamma(\frac{\alpha + 1}{2})}{\sqrt{\pi} \Gamma(1 - \frac{\alpha}{2})}$
and P.V. represents the principal value integral:
\begin{equation*}
\textrm{P.V.} \int_{\mathbb{R}} \frac{\phi(x) - \phi(y)}{\left| x - y \right|^{1 + \alpha}} dy
= \lim\limits_{\epsilon \rightarrow 0}\int_{\mathbb{R}\setminus B_{\epsilon}(x)}
\frac{\phi(x) - \phi(y)}{\left| x - y \right|^{1 + \alpha}} dy.
\end{equation*}
Here $\Gamma(\cdot)$ is the Gamma function and $B_{\epsilon}(x)$ is a ball of radius $\epsilon$ centered at $x$.
\end{definition}
Denote $L_{x,2}^{0}(\mathbb{R}) = \Big\{ u: \mathbb{R} \times [0,T] \rightarrow \mathbb{R} \mid \forall t \in [0,T],
u(\cdot,t) \in L^{2}(\mathbb{R}) ~\textrm{and}~
u(\cdot,t) = 0~\textrm{a.e.~in}~ \mathbb{R} \setminus \Omega \Big\}$ and
$L_{x,4}^{0}(\mathbb{R}) = \Big\{ u: \mathbb{R} \times [0,T] \rightarrow \mathbb{R} \mid \forall t \in [0,T],
u(\cdot,t) \in L^{4}(\mathbb{R}) ~\textrm{and}~
u(\cdot,t) = 0~\textrm{a.e.~in}~ \mathbb{R} \setminus \Omega \Big\}$.
For any $w,u \in L_{x,2}^{0}(\mathbb{R})$, their inner product is a function of  $t$ and defined by
\begin{equation*}
\langle w,u \rangle_x = \int_{\Omega} w(\xi,t) u(\xi,t) d\xi, \quad \forall t \in [0,T].
\end{equation*}
Correspondingly, the norm of $u \in L_{x,2}^{0}(\mathbb{R})$ can be defined as $\| u \|_{x,2}^2 = \langle u,u \rangle_x$.
For each $u \in L_{x,4}^{0}(\mathbb{R})$, the norm is given by
\begin{equation*}
\| u \|_{x,4}^4 = \int_{\Omega} \left| u(\xi,t) \right|^4 d\xi, \quad \forall t \in [0,T].
\end{equation*}
With these notations, the energy of model \eqref{eq1.1} can be defined as
\begin{equation}
E(\phi) = \frac{1}{4} \| \phi \|_{x,4}^{4} - \frac{1}{2} \| \phi \|_{x,2}^{2} + \frac{1}{4} | \Omega |
+ \frac{\varepsilon^2}{2} \left\langle (-\Delta)^{\alpha/2} \phi(x),\phi(x) \right\rangle_x.
\label{eq2.2}
\end{equation}
% Remark 1
\begin{remark}\label{remark1}
According to \cite{cai2019riesz}, we know that the fractional Laplacian and Riesz fractional derivative are equivalent,
i.e., if $\phi(x) \in \mathcal{S}(\mathbb{R})$, it holds that
\begin{equation*}
-(-\Delta)^{\alpha/2} \phi(x) = \frac{\partial^\alpha \phi(x)}{\partial |x|^\alpha}
=
%\frac{\partial^\alpha \phi(x)}{\partial |x|^\alpha} =
-\frac{1}{2 \cos(\frac{\pi \alpha}{2})\Gamma(2 - \alpha)}
\frac{d^2}{d x^2} \int_{-\infty}^{\infty} \left| x - \eta \right|^{1 - \alpha} \phi(\eta) d\eta,
\quad 1 < \alpha < 2, ~ x \in \Omega.
\end{equation*}
Then, we have \cite{macias2017structure}
\begin{equation*}
\left\langle -\frac{\partial^\alpha \phi(x)}{\partial |x|^\alpha},\phi(x) \right\rangle_x
= \left\langle \Xi^\alpha \phi(x),\Xi^\alpha \phi(x) \right\rangle_x,
\end{equation*}
where $\Xi^\alpha$ is the unique square root operator of $\frac{\partial^\alpha}{\partial |x|^\alpha}$.
With this in mind, the energy can be rewritten as
\begin{equation*}
E(\phi) = \frac{1}{4} \| \phi \|_{x,4}^{4} - \frac{1}{2} \| \phi \|_{x,2}^{2} + \frac{1}{4} | \Omega |
+ \frac{\varepsilon^2}{2} \| \Xi^\alpha \phi \|_{x,2}^2.
\end{equation*}
\end{remark}
Throughout this paper, unless otherwise specified, $C$ with and without subscript represent some positive constants.
%%%%%%%%%%%%%%%%%%%%%%%%%%%%%%%%%%%%%%%%%%%%%%%%%%
\section{The numerical method for the SFCH model and its analysis}
\label{sec3}

In this section, the modified BDF2 scheme (mBDF2) in combination with a finite difference method is presented.
Its energy stability and convergence are proved.
\subsection{The fully discrete numerical scheme}
\label{sec3.1}

For two given positive integers $M$ and $N$, let $\tau = \frac{T}{M}$ and $h = \frac{2 L}{N}$
be the time step size and the spatial grid size, respectively. Then, the domain $\bar{\Omega} \times [0,T]$ can be
covered by the mesh $\bar{\omega}_{h \tau} = \bar{\omega}_{h} \times \bar{\omega}_{\tau}$,
where
$\bar{\omega}_{h} = \{ x_i = -L + ih, ~i = 0, 1, \cdots, N \}$ and
$\bar{\omega}_{\tau} = \{ t_j = j \tau, ~j = 0, 1, \cdots, M \}$.

In \cite{duo2018novel}, the authors rewrote \eqref{eq2.1} as the weighted integral of a weaker singular function
and approximated it by the weighted trapezoidal rule.
A vital step in their discretization is that a so-called splitting parameter $\gamma \in (\alpha,2]$ 
was introduced ($\gamma = 1 + \frac{\alpha}{2}$ in this work),
which plays an important role in the accuracy. According to their proposal,
the fractional Laplacian \eqref{eq2.1} at the node $(x_i, t_j)$ can be approximated as
\begin{equation}
(-\Delta)^{\alpha/2} \phi(x_i, t_j) = c_{h}^{(\alpha,\gamma)} \sum\limits_{k = 1}^{N - 1} g_{i - k}^{(\alpha)} \phi(x_k, t_j)
+ \mathcal{O}(h^p) =  -\delta_h^\alpha \phi(x_i, t_j) + \mathcal{O}(h^p),
\label{eq3.1}
\end{equation}
where $c_{h}^{(\alpha,\gamma)} = \frac{c_{1,\alpha}}{\nu h^\alpha}$ with $\nu = \gamma - \alpha$ and
\begin{equation*}
g_{k}^{(\alpha)} =
\begin{cases}
\sum\limits_{\ell = 1}^{N - 1} \frac{(\ell + 1)^\nu - (\ell - 1)^\nu}{\ell^\gamma}
+ \frac{N^\nu - (N - 1)^\nu}{N^\gamma} %+  \left( 2^\nu + \kappa_\gamma - 1 \right)
+ \frac{2 \nu}{\alpha N^\alpha}, & k = 0, \\
%-\frac{2^\nu}{2} \left( + \kappa_\gamma - 1 \right), & k = 1, \\
-\frac{(k + 1)^\nu - (k - 1)^\nu}{2 k^\gamma}, & k = 1,2,\cdots,N - 2, \\
g_{-k}^{(\alpha)}, & k = -1,-2, \cdots, 2 - N
\end{cases}
\end{equation*}
with the constant $\kappa_\gamma = 1$ for $\gamma = 1 + \frac{\alpha}{2}$.
The order $p \in (0,2]$ is determined by the spatial regularity of $\phi(x,t)$.
Let $\phi_i^j$ and $\mu_i^j$ represent the numerical approximations of $\phi(x_i, t_j)$ and $\mu(x_i, t_j)$, respectively.
Combining the space discretization \eqref{eq3.1} with the BDF2 formula, we get our mBDF2 scheme
for approximating Eq.~\eqref{eq1.1}
at the grid points $(x_i,t_j)~(1 \leq i \leq N - 1, 1 \leq j \leq M - 1)$:
\begin{equation}
\begin{cases}
\delta_t \phi_i^{j + 1} = \delta_h^\alpha \mu_i^{j + 1}, \\
\mu_i^{j + 1} = \left( \phi_i^{j + 1} \right)^3 - 2 \phi_i^{j} + \phi_i^{j - 1} - \varepsilon^2 \delta_h^\alpha \phi_i^{j + 1}
- \sigma \tau \delta_h^\alpha \left( \phi_i^{j + 1} - \phi_i^{j} \right),
\end{cases}
\label{eq3.2}
\end{equation}
%\begin{equation*}
%\begin{cases}
%\delta_t \phi_i^{j + 1} = \delta_h^\alpha \mu_i^{j + 1} + R_{i}^{j + 1}, \\
%\mu_i^{j + 1} = \left( \phi_i^{j + 1} \right)^3 - 2 \phi_i^{j} + \phi_i^{j - 1} - \varepsilon^2 \delta_h^\alpha \phi_i^{j + 1}
%- \sigma \tau \delta_h^\alpha \left( \phi_i^{j + 1} - \phi_i^{j} \right) + \tilde{R}_{i}^{j + 1},
%\end{cases}
%\end{equation*}
where $\delta_t \phi_i^{j + 1} = \frac{3 \phi_i^{j + 1} - 4 \phi_i^j + \phi_i^{j - 1}}{2 \tau}$.

Comparing with the standard BDF2 scheme for Eq.~\eqref{eq1.1},
the concave diffusion term \cite{yan2018second} $- 2 \phi_i^{j} + \phi_i^{j - 1}$ is updated explicitly.
Moreover, the Douglas-Dupont-type regularization term $\sigma \tau \delta_h^\alpha \left( \phi_i^{j + 1} - \phi_i^{j} \right)$
is added to keep the energy stability of \eqref{eq3.2}. The later analysis shows that $\sigma \geq 1/16$.
Note that for the scheme \eqref{eq3.2}, $\phi_{i}^{1}$ is unknown.
To preserve the temporal second-order accuracy, it is approximated by Taylor expansion using neighbouring function values
\begin{equation*}
\phi_i^{1} = \phi_i^{0} + \tau \delta_h^\alpha \mu_i^{0} \quad\textrm{with}\quad
\mu_i^{0} = \left( \phi_i^0 \right)^3 - \phi_i^{0} - \varepsilon^2 \delta_h^\alpha \phi_i^{0}.
\end{equation*}
Such an approximation has the accuracy $\mathcal{O}(\tau^2 + h^p)$.
%$j = 0$, $\phi_i^{j-1} = \phi_i^{-1}$ is defined outside of the interval $[0, T]$.
%To preserve the temporal second-order accuracy, it is approximated by Taylor expansion using neighbouring function values
%\begin{equation*}
%\phi_i^{-1} = \phi_i^{0} - \tau \delta_h^\alpha \mu_i^{0} \quad\textrm{with}\quad
%\mu_i^{0} = \left( \phi_i^0 \right)^3 - \phi_i^{0} - \varepsilon^2 \delta_h^\alpha \phi_i^{0}.
%\end{equation*}
%Such an approximation has the accuracy $\mathcal{O}(\tau^2 + h^p)$.

Next, the matrix form of the mBDF2 scheme \eqref{eq3.2} is stated.
Let $\bm{\phi}^j = \left[\phi_1^{j}, \phi_2^{j}, \cdots, \phi_{N - 1}^{j} \right]^T$,
$\bm{\mu}^j = \left[\mu_1^{j}, \mu_2^{j}, \cdots, \mu_{N - 1}^{j} \right]^T$ and
\begin{equation*}
G = -c_{h}^{(\alpha,\gamma)}
\begin{bmatrix}
g_0^{(\alpha)} & g_{-1}^{(\alpha)} & g_{-2}^{(\alpha)} & \cdots & g_{3 - N}^{(\alpha)}
& g_{2 - N}^{(\alpha)} \\
g_1^{(\alpha)} & g_0^{(\alpha)} & g_{-1}^{(\alpha)} & g_{-2}^{(\alpha)} & \cdots
& g_{3 - N}^{(\alpha)} \\
\vdots & g_1^{(\alpha)} & g_0^{(\alpha)} & \ddots & \ddots & \vdots \\
\vdots & \ddots & \ddots & \ddots & \ddots & g_{-2}^{(\alpha)} \\
g_{N - 3}^{(\alpha)} & \ddots & \ddots & \ddots & g_{0}^{(\alpha)} & g_{- 1}^{(\alpha)} \\
g_{N - 2}^{(\alpha)} & g_{N - 3}^{(\alpha)} & \cdots & \cdots & g_{1}^{(\alpha)}
& g_{0}^{(\alpha)}
\end{bmatrix}.
\end{equation*}
Then, for $1 \leq j \leq M - 1$, the matrix form of  Eq.~\eqref{eq3.2} is given by
\begin{equation*}
\begin{cases}
\delta_t \bm{\phi}^{j + 1} = G \bm{\mu}^{j + 1}, \\
\bm{\mu}^{j + 1} = \left( \bm{\phi}^{j + 1} \right)^3 - 2 \bm{\phi}^{j} + \bm{\phi}^{j - 1} - \varepsilon^2 G \bm{\phi}^{j + 1}
- \sigma \tau G \left( \bm{\phi}^{j + 1} - \bm{\phi}^{j} \right), \\
\bm{\phi}^{1} = \bm{\phi}^{0} + \tau G \bm{\mu}^{0}, \\
\bm{\mu}^{0} = \left( \bm{\phi}^{0} \right)^3 - \bm{\phi}^{0} - \varepsilon^2 G \bm{\phi}^{0}.
\end{cases}
\end{equation*}
%%%%%%%%%%%%%%%%%%%%%%%%%%%%%%%%%%%%%%%%%%%%%%%%%%%%%%%%%%%%%%%%%%%%%%%%%%%%%%%%%%%%%%%%%%%%
\subsection{Analysis of the mBDF2 scheme}
\label{sec3.2}

Let $\mathring{\mathcal{V}}_h = \left\{ \bm{v} \in \mathbb{R}^{N + 1} | \bm{v} = \left( v_0,v_1, \cdots, v_{N} \right), v_0 = v_N = 0 \right\}$.
A discrete inner product and the corresponding norms are defined as
\begin{equation*}
(\bm{u},\bm{v}) = h \sum\limits_{k = 0}^{N} u_k v_k, \quad
\| \bm{u} \|_p^p = ( \left| \bm{u} \right|^p, \bm{1})~(1 \leq p < \infty) \quad \textrm{and} \quad
\| \bm{u} \|_\infty = \max\limits_{0 \leq k \leq N} \left| u_k \right| \quad \textrm{for}~
\forall \bm{u},\bm{v} \in \mathring{\mathcal{V}}_h.
\end{equation*}
Here $\left| \bm{u} \right|^p = \left( \left|  u_0 \right|^p, \left|  u_1 \right|^p, \cdots, \left|  u_N \right|^p \right)$
and $\bm{1} = \left( 1, 1, \cdots, 1 \right)$.

As a first main result, we prove the invertibility of $-G$ in the next lemma.
% Lemma 3.1
\begin{lemma}\label{lemma3.1}
The matrix $-G$ is a symmetric positive definite Toeplitz matrix.
\end{lemma}
\begin{proof}
According to the definition of $G$, it is obvious that $-G$ is a symmetric Toeplitz matrix.
We only need to prove that all eigenvalues of $-G$ are positive.
Let $\lambda$ be an eigenvalue of $-G$. Then, based on Gerschgorin's circle theorem \cite{varga2004gervsgorin},
the $i$th Gershgorin disc of $-G$ is centered at $c_{h}^{(\alpha,\gamma)} g_0^{(\alpha)} > 0$
with radius
\begin{equation*}
r_i = c_{h}^{(\alpha,\gamma)} \sum_{\ell = i + 1 - N, \ell \neq 0}^{i - 1} \left| g_\ell^{(\alpha)} \right|
\leq c_{h}^{(\alpha,\gamma)}
 \sum\limits_{\ell = 1}^{N - 2} \frac{(\ell + 1)^\nu - (\ell - 1)^\nu}{\ell^\gamma}
< c_{h}^{(\alpha,\gamma)} g_0^{(\alpha)}.
\end{equation*}
Consequently, the matrix $-G$ is a symmetric positive definite Toeplitz matrix.
\end{proof}
On the other hand, for $\bm{u},\bm{v} \in \mathring{\mathcal{V}}_h$,
it holds $(\bm{u}, -\delta_h^\alpha \bm{v}) = -h \bm{u}^T G \bm{v}$.
Combining this with Lemma \ref{lemma3.1} shows that
$-\delta_h^\alpha$ is a positive self-adjoint operator on the Hilbert space $\mathring{\mathcal{V}}_h$.
It has a unique square root operator denoted by $\tilde{\delta}_h$ such that
$(\bm{u}, -\delta_h^\alpha \bm{v}) = (-\delta_h^{\alpha}\bm{u}, \bm{v})
= (\tilde{\delta}_h\bm{u}, \tilde{\delta}_h\bm{v})$.
Moreover, for each $\bm{u},\bm{v} \in \mathring{\mathcal{V}}_h$,
denote $(\bm{u},\bm{v})_{-1} := ((-\delta_h^{\alpha})^{-1}\bm{u}, \bm{v}) = (\bm{u}, (-\delta_h^{\alpha})^{-1} \bm{v})$
and $\| \bm{u} \|_{-1}^2 = (\bm{u},\bm{u})_{-1}$, where $(-\delta_h^{\alpha})^{-1}$ represents the inverse operator of $-\delta_h^{\alpha}$. For any grid function $\bm{\phi} \in \mathring{\mathcal{V}}_h$,
the discrete energy can be written as
\begin{equation}
E_h(\bm{\phi}) = \frac{1}{4} \| \bm{\phi} \|_{4}^{4} - \frac{1}{2} \| \bm{\phi} \|_{2}^{2}
+ \frac{1}{4} | \Omega |
+ \frac{\varepsilon^2}{2} \| \tilde{\delta}_h \bm{\phi}\|_2^2.
\label{eq3.3a}
\end{equation}
Now, we can show the energy stability and convergence of the mBDF2 scheme \eqref{eq3.2}.
% Theorem 3.1
\begin{theorem}\label{th3.1}
Let $\left\{ \bm{\phi}^{k} \right\}_{k = 0}^{M}$ be a solution of \eqref{eq3.2}.
For $k \geq 0$, define the modified energy as
\begin{equation*}
\mathcal{E}_h(\bm{\phi}^{k + 1}, \bm{\phi}^{k})
= E_h(\bm{\phi}^{k + 1}) + \frac{1}{4 \tau} \| \bm{\phi}^{k + 1} - \bm{\phi}^{k} \|_{-1}^2
+ \frac{1}{2} \| \bm{\phi}^{k + 1} - \bm{\phi}^{k} \|_{2}^2.
\end{equation*}
Then, for $\sigma \geq 1/16$, the modified energy $\mathcal{E}_h(\bm{\phi}^{k + 1}, \bm{\phi}^{k})$
is decaying for the mBDF2 scheme \eqref{eq3.2}, i.e.,
\begin{equation*}
\mathcal{E}_h(\bm{\phi}^{k + 1}, \bm{\phi}^{k})
\leq \mathcal{E}_h(\bm{\phi}^{k}, \bm{\phi}^{k - 1}), \quad k \geq 1.
\end{equation*}
\end{theorem}
\begin{proof}
Taking the inner product of Eq.~\eqref{eq3.2} with
$(-\delta_h^{\alpha})^{-1} \left( \bm{\phi}^{k + 1} - \bm{\phi}^{k} \right)$,
we have:
\begin{equation}
\begin{split}
& \Big( \delta_t \bm{\phi}^{k + 1}, (-\delta_h^{\alpha})^{-1} \left( \bm{\phi}^{k + 1} - \bm{\phi}^{k} \right) \Big)
= \frac{1}{2 \tau} \left( 3 \| \bm{\phi}^{k + 1} - \bm{\phi}^{k} \|_{-1}^2
- \left( \bm{\phi}^{k} - \bm{\phi}^{k - 1}, \bm{\phi}^{k + 1} - \bm{\phi}^{k} \right)_{-1} \right) \\
&\quad = \frac{1}{4 \tau} \left( 5 \| \bm{\phi}^{k + 1} - \bm{\phi}^{k} \|_{-1}^2
- \| \bm{\phi}^{k} - \bm{\phi}^{k - 1} \|_{-1}^2 + \| \bm{\phi}^{k + 1} - 2 \bm{\phi}^{k} + \bm{\phi}^{k - 1} \|_{-1}^2 \right) \\
&\quad \geq \frac{1}{4 \tau} \left( 5 \| \bm{\phi}^{k + 1} - \bm{\phi}^{k} \|_{-1}^2
- \| \bm{\phi}^{k} - \bm{\phi}^{k - 1} \|_{-1}^2 \right),
\end{split}
\label{eq3.3}
\end{equation}
and
\begin{equation}
\begin{split}
& \Big( -\delta_h^{\alpha} \left( \bm{\phi}^{k + 1} \right)^3,
(-\delta_h^{\alpha})^{-1} \left( \bm{\phi}^{k + 1} - \bm{\phi}^{k} \right) \Big)
= \left( \left( \bm{\phi}^{k + 1} \right)^3,\bm{\phi}^{k + 1} - \bm{\phi}^{k} \right)
= \| \bm{\phi}^{k + 1} \|_{4}^{4} - \left( \left( \bm{\phi}^{k + 1} \right)^3,\bm{\phi}^{k} \right) \\
&\quad = \frac{1}{4} \left( \| \bm{\phi}^{k + 1} \|_{4}^{4} - \| \bm{\phi}^{k} \|_{4}^{4} \right)
+  \frac{1}{4} \| \left( \bm{\phi}^{k + 1} \right)^2 - \left( \bm{\phi}^{k} \right)^2 \|_{2}^{2}
+  \frac{1}{2} \| \bm{\phi}^{k + 1} \left( \bm{\phi}^{k + 1} - \bm{\phi}^{k} \right) \|_{2}^{2} \\
&\quad \geq \frac{1}{4} \left( \| \bm{\phi}^{k + 1} \|_{4}^{4} - \| \bm{\phi}^{k} \|_{4}^{4} \right),
\end{split}
\label{eq3.4}
\end{equation}
where the identity $\left( -\delta_h^{\alpha}\bm{u}, (-\delta_h^{\alpha})^{-1}\bm{v} \right) = \left( \bm{u}, \bm{v} \right)$ is used.
Furthermore,
\begin{equation}
\begin{split}
& \Big( -\delta_h^{\alpha} \left( -2 \bm{\phi}^k  + \bm{\phi}^{k - 1} \right),
(-\delta_h^{\alpha})^{-1} \left( \bm{\phi}^{k + 1} - \bm{\phi}^{k} \right) \Big)
= - \left( 2 \bm{\phi}^k  - \bm{\phi}^{k - 1},\bm{\phi}^{k + 1} - \bm{\phi}^{k} \right) \\
&\quad= - \left( \bm{\phi}^k,\bm{\phi}^{k + 1} - \bm{\phi}^{k} \right)
- \left( \bm{\phi}^k  - \bm{\phi}^{k - 1},\bm{\phi}^{k + 1} - \bm{\phi}^{k} \right) \\
&\quad = -\frac{1}{2} \left( \| \bm{\phi}^{k + 1} \|_{2}^{2} - \| \bm{\phi}^{k} \|_{2}^{2} \right)
-  \frac{1}{2} \| \bm{\phi}^{k} - \bm{\phi}^{k - 1} \|_{2}^{2}
+  \frac{1}{2} \| \bm{\phi}^{k + 1} - 2 \bm{\phi}^{k} + \bm{\phi}^{k - 1} \|_{2}^{2} \\
&\quad \geq -\frac{1}{2} \left( \| \bm{\phi}^{k + 1} \|_{2}^{2} - \| \bm{\phi}^{k} \|_{2}^{2} \right)
-  \frac{1}{2} \| \bm{\phi}^{k} - \bm{\phi}^{k - 1} \|_{2}^{2},
\end{split}
\label{eq3.5}
\end{equation}
\begin{equation}
\begin{split}
\Big( \left( -\delta_h^{\alpha} \right)^2 \bm{\phi}^{k + 1},
(-\delta_h^{\alpha})^{-1} \left( \bm{\phi}^{k + 1} - \bm{\phi}^{k} \right) \Big)
& = \Big( - \delta_h^{\alpha} \bm{\phi}^{k + 1},\bm{\phi}^{k + 1} \Big)
- \Big( -\delta_h^{\alpha} \bm{\phi}^{k + 1}, \bm{\phi}^{k} \Big) \\
& = \| \tilde{\delta}_h \bm{\phi}^{k + 1} \|_{2}^{2}
- \left( \tilde{\delta}_h \bm{\phi}^{k + 1},\tilde{\delta}_h \bm{\phi}^{k} \right) \\
& = \frac{1}{2} \left( \| \tilde{\delta}_h \bm{\phi}^{k + 1} \|_{2}^{2}
- \| \tilde{\delta}_h \bm{\phi}^{k} \|_{2}^{2} \right)
+  \frac{1}{2} \| \tilde{\delta}_h \left( \bm{\phi}^{k + 1} -  \bm{\phi}^{k} \right) \|_{2}^{2} \\
& \geq \frac{1}{2} \left( \| \tilde{\delta}_h \bm{\phi}^{k + 1} \|_{2}^{2}
- \| \tilde{\delta}_h \bm{\phi}^{k} \|_{2}^{2} \right),
\end{split}
\label{eq3.6}
\end{equation}
and
\begin{equation}
\begin{split}
\Big( \left( -\delta_h^{\alpha} \right)^2 \left( \bm{\phi}^{k + 1} -  \bm{\phi}^{k} \right),
(-\delta_h^{\alpha})^{-1} \left( \bm{\phi}^{k + 1} - \bm{\phi}^{k} \right) \Big)
& = \Big( - \delta_h^{\alpha} \left( \bm{\phi}^{k + 1} -  \bm{\phi}^{k} \right),
\bm{\phi}^{k + 1} - \bm{\phi}^{k} \Big) \\
& = \| \tilde{\delta}_h \left( \bm{\phi}^{k + 1} -  \bm{\phi}^{k} \right) \|_{2}^{2}.
\end{split}
\label{eq3.7}
\end{equation}
Additionally, using the inequalities of Young and Cauchy-Schwarz, we get
\begin{equation}
\begin{split}
\frac{1}{\tau} \| \bm{\phi}^{k + 1} - \bm{\phi}^{k} \|_{-1}^{2}
+ \sigma \tau \| \tilde{\delta}_h \left( \bm{\phi}^{k + 1} -  \bm{\phi}^{k} \right) \|_{2}^{2}
& \geq 2 \sigma^{1/2} \| \bm{\phi}^{k + 1} - \bm{\phi}^{k} \|_{-1}
\| \tilde{\delta}_h \left( \bm{\phi}^{k + 1} -  \bm{\phi}^{k} \right) \|_{2} \\
& \geq 2 \sigma^{1/2} \| \bm{\phi}^{k + 1} -  \bm{\phi}^{k} \|_{2}^{2}.
\end{split}
\label{eq3.8}
\end{equation}
Combining Eqs.~\eqref{eq3.3}-\eqref{eq3.8} and using \eqref{eq3.2}, one obtains
\begin{equation*}
\begin{split}
& \mathcal{E}_h(\bm{\phi}^{k + 1}, \bm{\phi}^{k}) - \mathcal{E}_h(\bm{\phi}^{k}, \bm{\phi}^{k - 1}) \\
& = E_h(\bm{\phi}^{k + 1}) - E_h(\bm{\phi}^{k})
+ \frac{1}{4 \tau} \left( \| \bm{\phi}^{k + 1} - \bm{\phi}^{k} \|_{-1}^{2}
- \| \bm{\phi}^{k} - \bm{\phi}^{k - 1} \|_{-1}^{2} \right)
+ \frac{1}{2} \left( \| \bm{\phi}^{k + 1} - \bm{\phi}^{k} \|_{2}^{2}
- \| \bm{\phi}^{k} - \bm{\phi}^{k - 1} \|_{2}^{2} \right) \\
& = \frac{1}{4 \tau} \left( 5 \| \bm{\phi}^{k + 1} - \bm{\phi}^{k} \|_{-1}^{2}
- \| \bm{\phi}^{k} - \bm{\phi}^{k - 1} \|_{-1}^{2} \right)
- \frac{1}{\tau} \| \bm{\phi}^{k + 1} - \bm{\phi}^{k} \|_{-1}^{2}
+ \frac{1}{4} \left( \| \bm{\phi}^{k + 1} \|_{4}^{4} - \| \bm{\phi}^{k} \|_{4}^{4} \right) \\
&\quad + \left[ -\frac{1}{2} \left( \| \bm{\phi}^{k + 1} \|_{2}^{2} - \| \bm{\phi}^{k} \|_{2}^{2} \right)
- \frac{1}{2} \| \bm{\phi}^{k + 1} - \bm{\phi}^{k} \|_{2}^{2} \right]
+ \frac{\varepsilon^2}{2} \left( \| \tilde{\delta}_h \bm{\phi}^{k + 1} \|_{2}^{2}
- \| \tilde{\delta}_h \bm{\phi}^{k} \|_{2}^{2} \right)
+ \frac{1}{2} \| \bm{\phi}^{k + 1} -  \bm{\phi}^{k} \|_{2}^{2} \\
& \leq -\left( \frac{1}{\tau} \| \bm{\phi}^{k + 1} - \bm{\phi}^{k} \|_{-1}^{2}
+ \sigma \tau \| \tilde{\delta}_h \left( \bm{\phi}^{k + 1} -  \bm{\phi}^{k} \right) \|_{2}^{2} \right)
+ \frac{1}{2} \| \bm{\phi}^{k + 1} -  \bm{\phi}^{k} \|_{2}^{2} \\
& \leq \left( -2 \sigma^{1/2} + \frac{1}{2} \right) \| \bm{\phi}^{k + 1} -  \bm{\phi}^{k} \|_{2}^{2}
\leq 0, \qquad \textrm{if}~\sigma \geq 1/16.
\end{split}
\end{equation*}
This proves the desired result.
\end{proof}
% Remark 2
\begin{remark}\label{remark2}
Consider Eq.~\eqref{eq1.1} with the fractional Laplacian $-(-\Delta)^{\alpha/2}$ replaced by
the Riesz fractional derivative $\frac{\partial^\alpha}{\partial |x|^\alpha}$ (see Remark \ref{remark1}).
For its discretization, consider the fractional centered difference formula \cite{ortigueira2006riesz}.
Then, the energy stability of the resulting scheme can be proved analogously to
the proof of Theorem \ref{th3.1}.
\end{remark}
As a consequence of Theorem \ref{th3.1}, the next corollary provides
a uniform bound for the mBDF2 scheme \eqref{eq3.2}.
We denote by $C_{x,t}^{(s, \alpha/2),3} \left( \Omega \times (0,T] \right)$ the H\"{o}lder space of
functions with spatial regularity $C^{s, \alpha/2} \left( \Omega \right)$
(see \cite{duo2018novel} for the further definition of this H\"{o}lder space)
and temporal regularity $C^{3} (0,T]$.
% Corollary 1
\begin{corollary}\label{coro1}
Let $\sigma \geq 1/16$ and assume that the initial value $\phi_0(x)$ is sufficiently regular
and satisfies 
\begin{equation*}
E_h(\bm{\phi}^1) + \frac{\tau}{4} \| \tilde{\delta}_h \bm{\mu}^0 \|_2^2
+ \frac{\tau^2}{2} \| \delta_h^{\alpha} \bm{\mu}^0 \|_2^2 \leq C_0
\end{equation*}
for some $C_0$ (independent of $h$).
Then, the numerical solution, given by Eq.~\eqref{eq3.2}, satisfies the uniform bound:
%Then, we have the following uniform
%$C_{x,t}^{(s, \alpha/2),2} \left( \Omega \times (0,T] \right)~(s \geq 1)$ bound for Eq.~\eqref{eq3.2}:
\begin{equation}
\| \bm{\phi}^k \|_\infty \leq C_1~( 0 \leq k\tau \leq T ),
\label{eq3.9}
\end{equation}
where $C_1$ is independent of $h$, $\tau$ and $T$.
\end{corollary}
\begin{proof}
Theorem \ref{th3.1} shows that
\begin{equation}
\begin{split}
E_h(\bm{\phi}^k) \leq \mathcal{E}_h(\bm{\phi}^{k}, \bm{\phi}^{k - 1})
& \leq \mathcal{E}_h(\bm{\phi}^{1}, \bm{\phi}^{0})
= E_h(\bm{\phi}^{1}) + \frac{1}{4 \tau} \| \bm{\phi}^{1} - \bm{\phi}^{0} \|_{-1}^2
+ \frac{1}{2} \| \bm{\phi}^{1} - \bm{\phi}^{0} \|_{2}^2 \\
& = E_h(\bm{\phi}^1) + \frac{\tau}{4} \| \tilde{\delta}_h \bm{\mu}^0 \|_2^2
+ \frac{\tau^2}{2} \| \delta_h^{\alpha} \bm{\mu}^0 \|_2^2 \leq C_0.
\end{split}
\label{eq3.10}
\end{equation}
Substituting the trivial estimate
\begin{equation*}
\frac{1}{4} \| \bm{\phi}^k \|_{4}^{4} - \| \bm{\phi}^k \|_{2}^{2} \geq - | \Omega |
\end{equation*}
into Eq.~\eqref{eq3.3a} and using the bound \eqref{eq3.10} gives
\begin{equation*}
\| \bm{\phi}^k \|_{2}^{2} + \varepsilon^2 \| \tilde{\delta}_h \bm{\phi}^k \|_{2}^{2}
\leq 2 \left( C_0 + \frac{3}{4} | \Omega | \right).
\end{equation*}
Then, by a discrete Sobolev imbedding theorem, one has
\begin{equation*}
\| \bm{\phi}^k \|_{\infty}^{2}
\leq C_2 \left( \| \bm{\phi}^k \|_{2}^{2} + \| \tilde{\delta}_h \bm{\phi}^k \|_{2}^{2} \right)
\leq 2 C_2/\varepsilon^2 \left( C_0 + \frac{3}{4} | \Omega | \right).
\end{equation*}
With $C_1 = \varepsilon^{-1} \sqrt{2 C_0 C_2 + \frac{3 C_2}{2} |\Omega | }$,
the proof is completed.
\end{proof}
We next introduce an auxiliary norm for proceeding with the convergence analysis.
For $\vec{\bm{p}} = [\bm{u}, \bm{v}]^T \in \mathring{\mathcal{V}}_h \times \mathring{\mathcal{V}}_h$,
the weighted norm (called $B$-norm) is defined by
\begin{equation*}
\| \vec{\bm{p}} \|_{B}^{2} = \left( \vec{\bm{p}}, B \vec{\bm{p}} \right)
= \left( B \vec{\bm{p}}, \vec{\bm{p}} \right), \qquad
B =
\begin{bmatrix}
\frac{1}{2} & -1 \\
-1 & \frac{5}{2}
\end{bmatrix}.
\end{equation*}
This norm is well-defined because $B$ is symmetric and positive.
In addition, $\| \vec{\bm{p}} \|_{B}^{2} \geq \frac{1}{2} \| \bm{v} \|_{2}^{2}$
since
\begin{equation*}
B =
\begin{bmatrix}
\frac{1}{2} & -1 \\
-1 & 2
\end{bmatrix}
+ \begin{bmatrix}
0 & 0 \\
0 & \frac{1}{2}
\end{bmatrix} := B_1 + B_2,
\end{equation*}
where $B_1$ is symmetric positive semi-definite.
With the help of this norm, we now prove the convergence of the mBDF2 scheme \eqref{eq3.2}.
% Theorem 3.2
\begin{theorem}\label{th3.2}
Assume that $\Phi(x,t) \in C_{x,t}^{(s,\alpha/2),3} \left( \Omega \times [0,T] \right)$
is the exact solution of the SFCH model \eqref{eq1.1} and let
$\bm{\phi}^k$ be the solution of \eqref{eq3.2}.
Further, let $\bm{e}^k$ denote the error function, i.e. $\bm{e}_i^k = \Phi(x_i,t_k) - \bm{\phi}_i^k $
and let $\sigma \geq 1/16$. Then, for $\tau$ and $h$ are sufficiently small, we have
\begin{equation*}
\| \bm{e}^k \|_{2}^{2} + 2 \sigma \tau^2 \| \delta_{h}^{\alpha} \bm{e}^k \|_{2}^{2}
\leq C_3 \left( \tau^2 + h^p \right)^2~(1 \leq k \leq M),
\end{equation*}
where $C_3$ is independent of $\tau$ and $h$.
\end{theorem}
\begin{remark}
Note that the spatial order $p \leq 2$ depends on $s$.
The optimal rate $p = 2$ requires $s$ sufficiently large, see also \cite{duo2018novel}.
\end{remark}
\begin{proof}
The exact solution $\Phi(x,t)\in C_{x,t}^{(s,\alpha/2),3}(\Omega\times[0,T])$ satisfies the equation
\begin{equation*}
\begin{cases}
\partial_t \Phi(x_i,t_{k+1}) = -(-\Delta)^{\alpha/2}\tilde\mu(x_i,t_{k+1}) \\
\tilde\mu(x_i,t_{k+1}) = \Phi^3(x_i,t_{k+1}) - \Phi(x_i,t_{k+1}) + \varepsilon^2(-\Delta)^{\alpha/2}\Phi(x_i,t_{k+1}).
\end{cases}
\end{equation*}
Inserting the exact solution into the numerical scheme gives the perturbed scheme
\begin{equation}
\begin{cases}
\delta_t \bm{\Phi}^{k + 1} = \delta_h^\alpha \tilde{\bm\mu}^{k + 1} + D^{k + 1} \\
\tilde{\bm\mu}^{k + 1} = (\bm\Phi^{k + 1})^3 - 2\bm\Phi^{k} + \bm\Phi^{k - 1}
- \varepsilon^2\delta_h^\alpha\bm{\Phi}^{k + 1}
- \sigma \tau \delta_h^\alpha (\bm\Phi^{k + 1} - \bm\Phi^{k}) + \tilde D^{k + 1},
\end{cases}
\label{eq3.11}
\end{equation}
where $\bm\Phi^k = [\Phi(x_1,t_k),\Phi(x_2,t_k),\cdots,\Phi(x_{N-1},t_k)]^T$
and $\tilde{\bm\mu}^k = [\tilde{\mu}(x_1,t_k),\tilde{\mu}(x_2,t_k),\cdots,\tilde{\mu}(x_{N-1},t_k)]^T$.
The defects $D^{k+1}$ and $\tilde D^{k+1}$ are given by
\begin{equation*}
D^{k + 1} = \delta_t\bm\Phi^{k + 1} - \left( \partial_t \Phi(x_i,t_{k + 1}) \right)_{i = 1}^{N - 1}
- \delta_h^\alpha \tilde{\bm\mu}^{k + 1}  - \left( (-\Delta)^{\alpha/2} \tilde \mu(x_i,t_{k + 1}) \right)_{i = 1}^{N - 1}
\end{equation*}
and
\begin{equation*}
\tilde D^{k + 1}  =  -\bm\Phi^{k + 1} + 2 \bm\Phi^{k} - \bm\Phi^{k - 1}
+ \varepsilon^2 \left( \delta_h^\alpha \bm\Phi^{k + 1}
+ \left( (-\Delta)^{\alpha/2} \Phi(x_i,t_{k + 1}) \right)_{i = 1}^{N - 1} \right)
+ \sigma \tau \delta_h^\alpha (\bm\Phi^{k + 1} - \bm\Phi^{k}).
\end{equation*}
The techniques of \cite[Theorem~3.2]{duo2018novel} together with straightforward Taylor expansions
\begin{equation*}
\begin{split}
\delta_t \Phi(x_i,t_{k + 1})  - \partial_t \Phi(x_i,t_{k + 1})
&= 4 \int_0^\tau \tfrac{(\tau - r)^2}{2} \partial_{ttt} \Phi(x_i,t_{k + 1} - r) \,dr\\
& \qquad - \int_0^{2 \tau} \tfrac{(2 \tau - r)^2}{2} \partial_{ttt} \Phi(x_i,t_{k + 1} - r) \,dr\\
\left( -\bm\Phi^{k + 1} + 2 \bm\Phi^{k} - \bm\Phi^{k - 1} \right)_i
&= -\int_0^\tau \! \int_{-r}^r \partial_{tt} \Phi(x_i, t_k + \rho) \,d\rho dr, \\
\left( \bm\Phi^{k + 1} - \bm\Phi^{k} \right)_i &= \int_0^\tau \partial_t \Phi(x_i, t_k + r) \,dr
\end{split}
\end{equation*}
show that
\begin{equation*}
\|D^{k + 1}\|_2 = \mathcal{O}(\tau^2 + h^p) \qquad \text{and} \qquad
\|\delta_h^\alpha\tilde D^{k + 1}\|_2 = \mathcal{O}(\tau^2 + h^p).
\end{equation*}
%Then, a careful consistency analysis shows
%\begin{equation}
%\delta_t \bm{\Phi}^{k + 1} = \delta_{h}^{\alpha} \left[
%\left( \bm{\Phi}^{k + 1} \right)^3 - 2 \bm{\Phi}^k + \bm{\Phi}^{k - 1}
%-\varepsilon^2 \delta_{h}^{\alpha} \bm{\Phi}^{k + 1}
%- \sigma \tau \delta_{h}^{\alpha} \left( \bm{\Phi}^{k + 1} - \bm{\Phi}^{k} \right) \right] + \hat{R}^{k + 1},
%\end{equation}
%where $\| \hat{R}^{k + 1} \|_2 \leq C_5 (\tau^2 + h^p)$.
Subtracting Eq.~\eqref{eq3.2} from Eq.~\eqref{eq3.11}, we get the error recursion
\begin{equation*}
\delta_t \bm{e}^{k + 1} = \delta_{h}^{\alpha} \left[
\left( \bm{\Phi}^{k + 1} \right)^3 - \left( \bm{\phi}^{k + 1} \right)^3 - 2 \bm{e}^k + \bm{e}^{k - 1}
-\varepsilon^2 \delta_{h}^{\alpha} \bm{e}^{k + 1}
- \sigma \tau \delta_{h}^{\alpha} \left( \bm{e}^{k + 1} - \bm{e}^{k} \right) \right] + \hat{R}^{k + 1},
\end{equation*}
where $\| \hat{R}^{k + 1} \|_2 \leq C_4 (\tau^2 + h^p)$.
Taking the discrete inner product with $\bm{e}^{k + 1}$, one further gets
\begin{equation}
\begin{split}
& \left( \delta_t \bm{e}^{k + 1}, \bm{e}^{k + 1} \right)
+ \varepsilon^2 \left( \left( \delta_{h}^{\alpha} \right)^2 \bm{e}^{k + 1}, \bm{e}^{k + 1} \right)
+ \sigma \tau \left( \left( \delta_{h}^{\alpha} \right)^2 \left( \bm{e}^{k + 1} - \bm{e}^{k} \right), \bm{e}^{k + 1} \right) \\
&\quad = \left( \delta_{h}^{\alpha} \left[ \left( \bm{\Phi}^{k + 1} \right)^3 - \left( \bm{\phi}^{k + 1} \right)^3 \right],
\bm{e}^{k + 1} \right) + \left( \delta_{h}^{\alpha} \left( -2 \bm{e}^k + \bm{e}^{k - 1} \right),
\bm{e}^{k + 1} \right) + \left( \hat{R}^{k + 1}, \bm{e}^{k + 1} \right).
\end{split}
\label{eq3.12}
\end{equation}
For the inner products in Eq.~\eqref{eq3.12} have the following estimates:
\begin{equation}
\left( \delta_t \bm{e}^{k + 1}, \bm{e}^{k + 1} \right)
= \frac{1}{2 \tau} \left( \| \vec{\bm{e}}^k \|_{B}^{2} - \| \vec{\bm{e}}^{k - 1} \|_{B}^{2} \right)
+\frac{1}{4 \tau} \| \bm{e}^{k + 1} - 2 \bm{e}^{k} + \bm{e}^{k - 1} \|_{2}^{2}
\geq \frac{1}{2 \tau} \left( \| \vec{\bm{e}}^k \|_{B}^{2} - \| \vec{\bm{e}}^{k - 1} \|_{B}^{2} \right),
\label{eq3.13}
\end{equation}
where $\vec{\bm{e}}^k = \left[ \bm{e}^k, \bm{e}^{k + 1} \right]^T$ and
$\vec{\bm{e}}^{k - 1} = \left[ \bm{e}^{k - 1}, \bm{e}^{k} \right]^T$;
\begin{equation}\label{eq3.14}
\varepsilon^2 \left( \left( \delta_{h}^{\alpha} \right)^2 \bm{e}^{k + 1}, \bm{e}^{k + 1} \right)
= \varepsilon^2 \| \delta_{h}^{\alpha} \bm{e}^{k + 1} \|_{2}^{2};
\end{equation}
\begin{equation}\label{eq3.15}
\sigma \tau \left( \left( \delta_{h}^{\alpha} \right)^2 \left( \bm{e}^{k + 1} - \bm{e}^{k} \right), \bm{e}^{k + 1} \right)
= \sigma \tau \left( \delta_{h}^{\alpha} \left( \bm{e}^{k + 1} - \bm{e}^{k} \right), \delta_{h}^{\alpha} \bm{e}^{k + 1} \right)
\geq \frac{\sigma \tau}{2} \left( \| \delta_{h}^{\alpha} \bm{e}^{k + 1} \|_{2}^{2}
- \| \delta_{h}^{\alpha} \bm{e}^{k} \|_{2}^{2} \right);
\end{equation}
\begin{equation}\label{eq3.16}
\begin{split}
& \left( \delta_{h}^{\alpha} \left[ \left( \bm{\Phi}^{k + 1} \right)^3 - \left( \bm{\phi}^{k + 1} \right)^3 \right],
\bm{e}^{k + 1} \right) =
\epsilon_0 \| \left( \bm{\Phi}^{k + 1} \right)^3 - \left( \bm{\phi}^{k + 1} \right)^3 \|_{2}^{2}
+ \frac{1}{4 \epsilon_0} \| \delta_{h}^{\alpha} \bm{e}^{k + 1} \|_{2}^{2} \\
&\quad %\leq C_6 \left( \bm{e}^{k + 1}, \delta_{h}^{\alpha} \bm{e}^{k + 1} \right)
\leq C_5 \epsilon_0 \| \bm{e}^{k + 1} \|_{2}^{2}
+ \frac{1}{4 \epsilon_0} \| \delta_{h}^{\alpha} \bm{e}^{k + 1} \|_{2}^{2} ~(\epsilon_0 > 0),
\end{split}
\end{equation}
where Corollary \ref{coro1} is used;
\begin{equation}\label{eq3.17}
\begin{split}
\left( \delta_{h}^{\alpha} \left( -2 \bm{e}^k + \bm{e}^{k - 1} \right), \bm{e}^{k + 1} \right)
& = \left( -2 \bm{e}^k + \bm{e}^{k - 1} , \delta_{h}^{\alpha} \bm{e}^{k + 1} \right)
\leq \epsilon_1 \| 2 \bm{e}^k - \bm{e}^{k - 1} \|_{2}^{2}
+ \frac{1}{4 \epsilon_1} \| \delta_{h}^{\alpha} \bm{e}^{k + 1} \|_{2}^{2} \\
%\leq \sqrt{6} \| \delta_{h}^{\alpha} \bm{e}^{k + 1} \|_2 \sqrt{\| \bm{e}^k \|_{2}^{2} + \| \bm{e}^{k - 1} \|_{2}^{2}} \\
%\leq C_7 \| \vec{\bm{e}}^{k - 1} \|_{B} \| \delta_{h}^{\alpha} \bm{e}^{k + 1} \|_2 \\
& \leq \frac{\epsilon_1}{2} \| \vec{\bm{e}}^{k - 1} \|_{B}^{2}
+ \frac{1}{4 \epsilon_1} \| \delta_{h}^{\alpha} \bm{e}^{k + 1} \|_{2}^{2}  ~(\epsilon_1 > 0),
\end{split}
\end{equation}
where the Cauchy-Schwarz inequality and Young's inequality are employed;
\begin{equation}\label{eq3.18}
\left( \hat{R}^{k + 1}, \bm{e}^{k + 1} \right)
\leq \epsilon_2 \| \hat{R}^{k + 1} \|_{2}^{2} + \frac{1}{4 \epsilon_2} \| \bm{e}^{k + 1} \|_{2}^{2} ~(\epsilon_2 > 0).
\end{equation}
Substituting Eqs.~\eqref{eq3.13}-\eqref{eq3.18} into Eq.~\eqref{eq3.12} yields
\begin{equation*}
\begin{split}
W^{k + 1} - W^{k} +2 \tau \varepsilon^2 \| \delta_{h}^{\alpha} \bm{e}^{k + 1} \|_{2}^{2}
& \leq 2 \tau C_5 \epsilon_0 \| \bm{e}^{k + 1} \|_{2}^{2}
+ \tau \epsilon_1 \| \vec{\bm{e}}^{k - 1} \|_{B}^{2}
+2 \tau \epsilon_2 \| \hat{R}^{k + 1} \|_{2}^{2} \\
&\quad + \frac{\tau}{2 \epsilon_2} \| \bm{e}^{k + 1} \|_{2}^{2}
+ \frac{\tau}{2}\left( \frac{1}{\epsilon_0} + \frac{1}{\epsilon_1} \right)
\| \delta_{h}^{\alpha} \bm{e}^{k + 1} \|_{2}^{2} ,
\end{split}
\end{equation*}
where $W^k = \| \vec{\bm{e}}^{k - 1} \|_{B}^{2} + \sigma \tau^2 \| \delta_{h}^{\alpha} \bm{e}^{k} \|_{2}^{2}$.
Fixing $\epsilon_0 = \epsilon_1 = \frac{1}{2 \varepsilon^2}$ and $\epsilon_2 = \frac{\varepsilon^2}{2}$,
we arrive at
\begin{equation*}
W^{k + 1} - W^{k}
\leq \frac{C \tau}{\varepsilon^2} \left( \| \vec{\bm{e}}^{k} \|_{B}^{2}
+ \| \vec{\bm{e}}^{k - 1} \|_{B}^{2} \right) + \varepsilon^2 \tau \| \hat{R}^{k + 1} \|_{2}^{2}
\leq \frac{C}{\varepsilon^2} \tau \left( W^{k + 1} + W^{k} \right)
+ \varepsilon^2 \tau \| \hat{R}^{k + 1} \|_{2}^{2}.
\end{equation*}
According to Gronwall's inequality in \cite[Lemma 4.8]{ran2016conservative},
the above inequality implies that
\begin{equation*}
W^{k} \leq \exp\left( \frac{4 C}{\varepsilon^2} T \right)
\left[ W^1 + C_{4}^{2} T \varepsilon^2 \left( \tau^2 + h^p \right)^2 \right],
\end{equation*}
where $W^1 = \frac{5}{2} \| \bm{e}^{1} \|_{2}^{2} + \sigma \tau^2 \| \delta_{h}^{\alpha} \bm{e}^{1} \|_{2}^{2}$
with $\| \bm{e}^{1} \|_{2} = \mathcal{O} \left( \tau^2 + h^p \right)$.
On the other hand, $\| \delta_{h}^{\alpha} \bm{e}^{1} \|_{2}^{2}$ is bounded
since the initial data are sufficiently smooth in $x$.
The target result can be obtained immediately and the proof is completed.
\end{proof}
\section{The preconditioned iterative method}
\label{sec4}

For solving the nonlinear time-stepping scheme \eqref{eq3.2}, the two most common methods are
fixed point iteration and Newton's method.
In this work, we use Newton's method to handle this stiff nonlinear problem.
Let $\bm{\psi}^{j} = \left[ \bm{\phi}^j, \bm{\mu}^j \right]^T$ be the combined vector of unknowns.
Then, the mBDF2 scheme \eqref{eq3.2} can be rewritten in the following equivalent form
\begin{equation}
F(\bm{\psi}^{j + 1}) := \mathcal{M} \bm{\psi}^{j + 1} -
\begin{bmatrix}
4 \bm{\phi}^{j} - \bm{\phi}^{j - 1} \\
-2 \bm{\phi}^{j} + \bm{\phi}^{j - 1} + \sigma \tau G \bm{\phi}^{j} + \left( \bm{\phi}^{j + 1} \right)^3
\end{bmatrix}
= \bm{0},
\label{eq4.1}
\end{equation}
where
\begin{equation*}
\mathcal{M} =
\begin{bmatrix}
3 I & -2 \tau G \\
\left( \varepsilon^2 + \sigma \tau \right) G & I
\end{bmatrix}.
\end{equation*}
Here, $\bm{0}$ and $I$ denote the zero and the identity matrix of suitable sizes, respectively.
Given an initial guess $\bm{\psi}^{j + 1 (0)}$,
the solution can be obtained from the iteration process
\begin{equation}
\mathcal{J}^{j + 1 (\ell)} \Delta \bm{\psi}^{j + 1 (\ell)} = F(\bm{\psi}^{j + 1 (\ell)}), \quad
\bm{\psi}^{j + 1 (\ell + 1)} = \bm{\psi}^{j + 1 (\ell)} - \Delta \bm{\psi}^{j + 1 (\ell)},
\label{eq4.2}
\end{equation}
where \begin{equation*}
\mathcal{J}^{j + 1 (\ell)} = \frac{\partial F(\bm{\psi}^{j + 1 (\ell)})}{ \partial \bm{\psi}^{j + 1 (\ell)}}
=
\begin{bmatrix}
3 I & -2 \tau G \\
\left( \varepsilon^2 + \sigma \tau \right) G - 3\textrm{diag}\left( \left( \bm{\phi}^{j + 1 (\ell)} \right)^2 \right)  & I
\end{bmatrix}
\end{equation*}
is the Jacobian matrix.
We use the desired stopping criterion
$\frac{\| -\Delta \bm{\psi}^{j + 1 (\ell)} \|}{\| \bm{\psi}^{j + 1 (0)} \|} \leq \texttt{tol}$,
where $\| \cdot \|$ is the $2$-norm of a vector and $\texttt{tol}$ is the prescribed tolerance of Newton's method.

Obviously, $\mathcal{J}^{j + 1 (\ell)}$ is also a $2 \times 2$ block matrix.
It is well known that for $2 \times 2$ block matrices, the four most common block preconditioners are
block diagonal, block lower triangular, block upper triangular and block LDU \cite{southworth2020note}.
%However, all these preconditioners contain Schur complement $S$ and the computation cost of
%$S^{-1}$ is generally very expensive.
In this work the block lower triangular preconditioner is chosen
\begin{equation*}
P^{j + 1 (\ell)} =
\begin{bmatrix}
3 I & \bm{0} \\
\left( \varepsilon^2 + \sigma \tau \right) G - 3 \textrm{diag}\left( \left( \bm{\phi}^{j + 1 (\ell)} \right)^2 \right)
& S
\end{bmatrix}
\end{equation*}
with the Schur complement
\begin{equation*}
S = I + \frac{2}{3} \tau \left( \varepsilon^2 + \sigma \tau \right) G^2
- 2 \tau  \textrm{diag}\left( \left( \bm{\phi}^{j + 1 (\ell)} \right)^2 \right) G.
\end{equation*}

In our case, the invertibility of $S$ is very difficult to prove and the computation of $S^{-1}$ is very expensive.
To optimize the computational complexity,
a circulant matrix $\hat{S}$ is designed to replace $S$.
More precisely, our block lower triangular preconditioner is given as
\begin{equation*}
P_{L}^{j + 1 (\ell)} =
\begin{bmatrix}
3 I & \bm{0} \\
\left( \varepsilon^2 + \sigma \tau \right) G - 3\textrm{diag}\left( \left( \bm{\phi}^{j + 1 (\ell)} \right)^2 \right)
& \hat{S}
\end{bmatrix},
\end{equation*}
where
\begin{equation*}
\hat{S} =  I + \frac{2}{3} \tau \left( \varepsilon^2 + \sigma \tau \right) \left[ sk(G) \right]^2
- 2 \tau \bar{\phi}^{j + 1 (\ell)} sk(G), \qquad
\bar{\phi}^{j + 1 (\ell)} = \frac{1}{N - 1} \sum\limits_{k = 1}^{N - 1} \left( \phi_{k}^{j + 1 (\ell)} \right)^2.
\end{equation*}
Here, $sk(G)$ represents the Strang-type skew-circulant matrix with its first column given as
\begin{equation*}
-c_{h}^{(\alpha,\gamma)} \left[
 g_0^{(\alpha)},  g_1^{(\alpha)}, \cdots, g_{\left \lfloor \frac{N - 1} {2}\right \rfloor}^{(\alpha)},
- g_{N - \left \lfloor \frac{N - 1} {2}\right \rfloor - 2}^{(\alpha)}, \cdots, - g_{1}^{(\alpha)} \right]^T
~(\textrm{if}~N~\textrm{is~even})
\end{equation*}
and
\begin{equation*}
-c_{h}^{(\alpha,\gamma)} \left[
 g_0^{(\alpha)}, g_1^{(\alpha)}, \cdots, g_{\left \lfloor \frac{N - 1} {2}\right \rfloor - 1}^{(\alpha)}, 0 ,
 - g_{N - \left \lfloor \frac{N - 1} {2}\right \rfloor - 2}^{(\alpha)}, \cdots, - g_{1}^{(\alpha)} \right]^T
~(\textrm{if}~N~\textrm{is~odd}).
\end{equation*}
We check the invertibility of $P_{L}^{j + 1(\ell)}$ by using its spectral decomposition.
According to the work \cite{ng2004iterative, chan2007introduction},
the skew-circulant matrix $sk(G)$ has the spectral decomposition:
$sk(G) = (\Omega \mathcal{F})^{*} \Lambda^{sk} \mathcal{F} \Omega$,
where $\Omega = \textrm{diag} \left( 1, (-1)^{-\frac{1}{N - 1}}, \cdots, (-1)^{-\frac{N - 2}{N - 1}} \right)$,
$\Lambda^{sk}$ is a diagonal matrix containing all eigenvalues of $sk(G)$,
$\mathcal{F}$ is the discrete Fourier matrix and ``$*$" means the conjugate transpose operation.
Then, the decomposition of $\hat{S}$ is $ (\Omega \mathcal{F})^{*} \Lambda \mathcal{F} \Omega$
with $\Lambda =  I + \frac{2}{3} \tau \left( \varepsilon^2 + \sigma \tau \right) \left( \Lambda^{sk} \right)^2
- 2 \tau \bar{\phi}^{j + 1 (\ell)} \Lambda^{sk}$.
With the help of this decomposition, the following result is obtained immediately.
% Theorem 4.1
\begin{theorem}
The preconditioner $P_{L}^{j + 1 (\ell)}$ is invertible.
\label{th4.1}
\end{theorem}
\begin{proof}
Since $P_{L}^{j + 1 (\ell)}$ is a block lower triangular matrix, to check its invertibility is equivalent to prove that $\hat{S}$ is invertible.
Without loss of generality, we assume that $N$ is even. The odd case can be proved analogously.

\textbf{Step 1.} We first show that all eigenvalues of $sk(G)$ are real and negative.
From the definition of $sk(G)$ one gets at once that it is a symmetric matrix.
Thus, the eigenvalues of $sk(G)$ are real.
Next, we use Gerschgorin's circle theorem \cite{varga2004gervsgorin}.
All the Gershgorin discs of $sk(G)$ are centered at $-c_{h}^{(\alpha,\gamma)} g_0^{(\alpha)} < 0$
with radius
\begin{equation*}
r_{sk} = c_{h}^{(\alpha,\gamma)}
\left( \sum_{\ell = 1}^{\left \lfloor \frac{N - 1} {2}\right \rfloor} \left| g_\ell^{(\alpha)} \right|
+ \sum_{\ell = 1}^{N - \left \lfloor \frac{N - 1} {2}\right \rfloor - 2} \left| g_\ell^{(\alpha)} \right| \right)
\leq 2 c_{h}^{(\alpha,\gamma)} \sum_{\ell = 1}^{\left \lfloor \frac{N - 1} {2}\right \rfloor} \left| g_\ell^{(\alpha)} \right|
< c_{h}^{(\alpha,\gamma)} g_0^{(\alpha)}.
\end{equation*}
This implies at once that all eigenvalues of $sk(G)$ are negative.

\textbf{Step 2.} Combining \textbf{Step 1} and the definition of $\bar{\phi}^{j + 1 (\ell)}$,
the $k$th eigenvalue of $\hat{S}$ satisfies
\begin{equation*}
\Lambda_{k, k} =  1 + \frac{2}{3} \tau \left( \varepsilon^2 + \sigma \tau \right)
\left( \Lambda^{sk} \right)_{k, k}^2 - 2 \tau \bar{\phi}^{j + 1 (\ell)} \Lambda^{sk}_{k, k} > 1.
\end{equation*}
Thus, the circulant matrix $\hat{S}$ is invertible and the proof is completed.
\end{proof}

Unfortunately, it is difficult to theoretically investigate the eigenvalue distribution of $\left( P_{L}^{j + 1 (\ell)} \right)^{-1} \mathcal{J}^{j + 1 (\ell)}$,
but we still can work out some figures to illustrate eigenvalue
distributions of several specified preconditioned matrices in the next section.
Furthermore, the following estimate about $sk(G)$ is true.
% Theorem 2.4
\begin{theorem}
For $\alpha \in (1,2)$ and $\gamma = 1 + \frac{\alpha}{2}$, we have
\begin{equation*}
\frac{\| sk(G) - G \|_{\infty}}{\| G \|_{\infty}} < \left( \frac{3}{2} + \frac{2 \nu}{\alpha C_6 N^{\alpha}} \right)^{-1},
\end{equation*}
where $C_6 = \sum\limits_{\ell = 1}^{\infty} \frac{\left( \ell + 1 \right)^{\nu} - \left( \ell - 1 \right)^{\nu}}{\ell^{\gamma}}$
and $\nu = 1 - \frac{\alpha}{2}$. Note that $C_6$ only dependents on $\alpha$.
\label{th4.2}
\end{theorem}
\begin{proof}
First, we show that the following series is convergent:
\begin{equation*}
\sum\limits_{\ell = 1}^{\infty} \frac{\left( \ell + 1 \right)^{\nu} - \left( \ell - 1 \right)^{\nu}}{\ell^{\gamma}}
= \sum\limits_{\ell = 1}^{\infty} \left(\frac{1}{\ell} \right)^{\alpha}
\left[ \left( 1 + \frac{1}{\ell} \right)^{\nu} - \left( 1 - \frac{1}{\ell} \right)^{\nu} \right]
= \sum\limits_{\ell = 1}^{\infty} a_{\ell} b_{\ell},
\end{equation*}
where $a_{\ell} = \left(\frac{1}{\ell} \right)^{\alpha}$
and $b_{\ell} = \left( 1 + \frac{1}{\ell} \right)^{\nu} - \left( 1 - \frac{1}{\ell} \right)^{\nu}$.
On the one hand, the positive series $\sum\limits_{\ell = 1}^{\infty} a_{\ell}$ is the $p$-series and convergent since $\alpha > 1$.
On the other hand, consider the function $\hat{b}(x) = \left( 1 + \frac{1}{x} \right)^{\nu} - \left( 1 - \frac{1}{x} \right)^{\nu}~(x \geq 1)$.
The first order derivative is $\frac{d \hat{b}(x)}{d x} = -\frac{\nu}{x^2}
\left[ \left( 1 + \frac{1}{x} \right)^{\nu - 1} + \left( 1 - \frac{1}{x} \right)^{\nu - 1} \right] < 0~(x \geq 1)$.
Thus, the sequence $\{ b_{\ell} \}_{\ell = 1}^{\infty}$ is monotonically decreasing and bounded.
%According to the Abel's test (or Abel's criterion) \cite{bromwich2005introduction},
Then, we know that the series $\sum\limits_{\ell = 1}^{\infty} a_{\ell} b_{\ell}$ is convergent.
Let $\sum\limits_{\ell = 1}^{\infty} a_{\ell} b_{\ell} = C_6$. The constant $C_6$ only dependents on $\alpha$.

Now, we are in position to estimate
\begin{equation*}
\frac{\| sk(G) - G \|_{\infty}}{\| G \|_{\infty}} %= \frac{\| sk(G) - G \|_{1}}{\| G \|_{1}}
< \frac{ \sum\limits_{\ell = 1}^{N - \left \lfloor \frac{N - 1} {2}\right \rfloor - 2} \left| g_\ell^{(\alpha)} \right|
+ \sum\limits_{\ell = \left \lfloor \frac{N - 1} {2}\right \rfloor + 1}^{N - 2} \left| g_\ell^{(\alpha)} \right| }
{\sum\limits_{\ell = 0}^{N - 2} \left| g_\ell^{(\alpha)} \right|}
< \frac{\sum\limits_{\ell = 1}^{N - 2} \frac{\left( \ell + 1 \right)^{\nu} - \left( \ell - 1 \right)^{\nu}}{\ell^{\gamma}}}
 {\frac{3}{2}\sum\limits_{\ell = 1}^{N - 2} \frac{\left( \ell + 1 \right)^{\nu} - \left( \ell - 1 \right)^{\nu}}{\ell^{\gamma}}
 + \frac{2 \nu}{\alpha N^{\alpha}}}
< \left( \frac{3}{2} + \frac{2 \nu}{\alpha C_6 N^{\alpha}} \right)^{-1},
\end{equation*}
which completes the proof.
%After several simple manipulations, it has
%\begin{equation*}
%\frac{\| sk(G) - G \|_{\infty}}{\| G \|_{\infty}} %= \frac{\| sk(G) - G \|_{1}}{\| G \|_{1}}
%< \frac{ \sum\limits_{\ell = 1}^{N - \left \lfloor \frac{N - 1} {2}\right \rfloor - 2} \left| g_\ell^{(\alpha)} \right|
%+ \sum\limits_{\ell = \left \lfloor \frac{N - 1} {2}\right \rfloor + 1}^{N - 2} \left| g_\ell^{(\alpha)} \right| }
%{\sum\limits_{\ell = 0}^{N - 2} \left| g_\ell^{(\alpha)} \right|}
% < \frac{2 \sum\limits_{\ell = 1}^{N - 2} \left| g_\ell^{(\alpha)} \right|}
% {\sum\limits_{\ell = 0}^{N - 2} \left| g_\ell^{(\alpha)} \right|}
%< \frac{2}{3}.
%\end{equation*}
%The proof is completed.
\end{proof}

It is interesting that the Strang-type skew-circulant matrix can also
be used as a preconditioner for other Toeplitz-like systems, such as \cite{zhao2018limited}.
For the convenience of the reader, we summarize our Newton's method in Algorithm \ref{alg1}.
\begin{algorithm}[H]
	\caption{Solve $\bm{\psi}^{j + 1}$ from Eq.~\eqref{eq4.1}}
	\begin{algorithmic}[1]
		\STATE {Given maximum number of iterations $maxit$, tolerance $tol_{out}$ and initial vector $\bm{\psi}^{j + 1(0)}$}
		\FOR {$\ell = 1, \cdots, maxit$}
		\STATE {Solve $\mathcal{J}^{j + 1 (\ell)} \Delta \bm{\psi}^{j + 1 (\ell)} = F(\bm{\psi}^{j + 1 (\ell)})$ with a preconditioned Krylov subspace method}
		\STATE {$\bm{\psi}^{j + 1 (\ell + 1)} = \bm{\psi}^{j + 1 (\ell)} - \Delta \bm{\psi}^{j + 1 (\ell)}$}
		\IF {$\frac{\| -\Delta \bm{\psi}^{j + 1 (\ell)} \|}{\| \bm{\psi}^{j + 1 (0)} \|} \leq tol_{out}$}
		\STATE $\bm{\psi}^{j + 1} = \bm{\psi}^{j + 1 (\ell + 1)}$
		\STATE \textbf{break}
		\ENDIF
		\ENDFOR
	\end{algorithmic}
	\label{alg1}
\end{algorithm}
%we cannot use the method given in \cite{zhao2018limited}
%to analyze the spectrum of $sk(G)^{-1} G$, because the generating function of $G$
%is no longer in the Wiener class \cite{ng2004iterative, chan2007introduction}.

\section{Numerical experiments}
\label{sec5}

Two examples are reported in this section. Example 1 shows the time and space convergence orders of our method mBDF2 \eqref{eq3.2}.
The performance of our preconditioner in Section \ref{sec4} is displayed in Example 2.
In order to illustrate the efficiency of $sk(G)$, the Strang's circulant preconditioner (denoted as $s(G)$) \cite{ng2004iterative, chan2007introduction} is also tested.
The first column of $s(G)$ is
\begin{equation*}
-c_{h}^{(\alpha,\gamma)}\left[
g_0^{(\alpha)}, g_1^{(\alpha)}, \cdots, g_{\left \lfloor \frac{N - 1} {2}\right \rfloor}^{(\alpha)},
g_{N - \left \lfloor \frac{N - 1} {2}\right \rfloor - 2}^{(\alpha)}, \cdots, g_{1}^{(\alpha)} \right]^T~(\textrm{if}~N~\textrm{is~even})
\end{equation*}
and
\begin{equation*}
-c_{h}^{(\alpha,\gamma)} \left[
g_0^{(\alpha)}, g_1^{(\alpha)}, \cdots, g_{\left \lfloor \frac{N - 1} {2}\right \rfloor - 1}^{(\alpha)}, 0,
g_{N - \left \lfloor \frac{N - 1} {2}\right \rfloor - 2}^{(\alpha)}, \cdots, g_{1}^{(\alpha)} \right]^T~(\textrm{if}~N~\textrm{is~odd}).
\end{equation*}
Note that another preconditioner denoted as $\hat{P}_{L}^{j + 1 (\ell)}$ is obtained by just replacing $\hat{S}$ with
\begin{equation*}
\hat{S}_s =  I + \frac{2}{3} \tau \left( \varepsilon^2 + \sigma \tau \right) \left[ s(G) \right]^2
- 2 \tau \bar{\phi}^{j + 1 (\ell)} s(G)
\end{equation*}
in $P_{L}^{j + 1 (\ell)}$.
In this work, we choose the flexible generalized minimal residual method FGMRES(0) \cite{saad2003iterative} to solve \eqref{eq4.2}.
The iteration is terminated if the relative residual error satisfies $\frac{\| \bm{r}^{(k)} \|}{\| \bm{r}^{(0)} \|} \leq 10^{-12}$
or the iteration number is more than $1000$, where $\bm{r}^{(k)}$ denotes the residual vector in the $k$th iteration.
The initial guess is chosen as the zero vector.
In Algorithm \ref{alg1}, the initial vector is chosen as
\begin{equation*}
\bm{\psi}^{j + 1 (0)} =
\begin{cases}
\left[ \bm{\phi}^0, \bm{\mu}^0 \right]^T, & j = 1, \\
\left[ 2\bm{\phi}^{j} - \bm{\phi}^{j - 1}, 2\bm{\mu}^{j} - \bm{\mu}^{j - 1} \right]^T, & j \geq 2,
\end{cases}
\end{equation*}
the number of $maxit$ and $tol_{out}$ are fixed as $200$ and $10^{-12}$, respectively.
Some other notations that will appear later are collected here:
\begin{equation*}
Err_\infty(h, \tau) = \max_{0 \leq i \leq N} \mid e_{i}^{M} \mid, \quad
Err_2(h, \tau) = \| \bm{e}^{M} \|_{2},
\end{equation*}
\begin{equation*}
CO_{\infty, \tau} =\log_{\tau_1/ \tau_2} \frac{Err_\infty(h, \tau_1)}{Err_\infty(h, \tau_2)}, \quad
CO_{2, \tau} =\log_{\tau_1/ \tau_2} \frac{Err_2(h, \tau_1)}{Err_2(h, \tau_2)},
\end{equation*}
\begin{equation*}
CO_{\infty, h} =\log_{h_1/ h_2} \frac{Err_\infty(h_1, \tau)}{Err_\infty(h_2, \tau)}, \quad
CO_{2, h} =\log_{h_1/ h_2} \frac{Err_2(h_1, \tau)}{Err_2(h_2, \tau)}.
\end{equation*}
``Time" is the total CPU time in seconds for solving the nonlinear system \eqref{eq4.2}.
``BS" means that Eq. \eqref{eq4.2} is solved by MATLAB's backslash operator,
``$\mathcal{P}$" and ``$\mathcal{P}_s$" represent that the FGMRES(0) method
with preconditioners $P_{L}^{j + 1 (\ell)}$ and $\hat{P}_{L}^{j + 1 (\ell)}$ to solve \eqref{eq4.2}, respectively.
``Iter1"  and ``Iter2" represent the average numbers of iterations required by Newton's method
and the preconditioned FGMRES(0) (called PFGMRES(0)) method respectively. More precisely,
\begin{equation*}
\mathrm{Iter1} = \frac{1}{M} \sum\limits_{j = 0}^{M - 1} \mathrm{Iter1} (j) \quad \mathrm{and} \quad
\mathrm{Iter2} = \frac{1}{M} \sum\limits_{j = 0}^{M - 1} \sum\limits_{\ell = 1}^{\mathrm{Iter1}(j)}
\frac{\mathrm{Iter2(\ell)}}{\mathrm{Iter1}(j)},
\end{equation*}
where $\mathrm{Iter1}(j)$ represents the number of iterations required by Newton's method for solving Eq.~\eqref{eq4.1},
and $\mathrm{Iter2(\ell)}$ is the number of iterations required
by the PFGMRES(0) method in line 3 of Algorithm \ref{alg1}.

All experiments were performed on a Windows 10 (64 bit) PC-Intel(R) Core(TM) i7-8700k
CPU 3.20 GHz, 16 GB of RAM using MATLAB R2018b.

\subsection{Verification of the convergence rate}
\label{sec5.1}

In the following numerical example we verify the convergence of our scheme \eqref{eq3.2}.
Since the analytical solution of Eq.~\eqref{eq1.1} is difficult to obtain, we proceed as in \cite{wang2019finite}.
An exact solution is constructed artificially by adding appropriate nonzero source terms.

\noindent\textbf{Example 1.}
We consider an inhomogeneous version of Eq.~\eqref{eq1.1}
\begin{equation*}
\begin{cases}
\partial_t \phi(x,t) = -(-\Delta)^{\alpha/2} \mu(x,t) + f(x,t), & \textrm{in} \quad \Omega \times (0,T], \\
\mu(x,t) = \phi^3(x,t) - \phi(x,t) + \varepsilon^2 (-\Delta)^{\alpha/2} \phi(x,t) + \psi(x,t),
& \textrm{in} \quad \Omega \times (0,T], \\
\phi(x,t) = \mu(x,t) = 0, & \textrm{in} \quad \mathbb{R} \setminus \Omega \times (0,T],\\
\phi(x,0) = \phi_0(x), &  \textrm{in} \quad \Omega
\end{cases}
\end{equation*}
with $L = T = 1$, $\varepsilon^2 = 0.1$
and the source terms
\begin{equation*}
\begin{split}
& f(x,t) = \exp(t) \left\{ \left( 1 - x^2 \right)^{3 + \alpha/2}
+ \frac{2^{\alpha} \Gamma(\frac{\alpha + 1}{2}) \Gamma(3 + \alpha/2)}{\sqrt{\pi} \Gamma(3)}
\left[ 1 - 2 \left( \alpha + 1 \right) x^2 + \frac{\left( \alpha + 1 \right) \left( \alpha + 3 \right)}{3} x^4 \right] \right\} \\
& \psi(x,t) = \exp(t) \bigg\{ \left( 1 - x^2 \right)^{2 + \alpha/2}
- \frac{2^{\alpha} \Gamma(\frac{\alpha + 1}{2}) \Gamma(4 + \alpha/2)}{\sqrt{\pi} \Gamma(4)}
\varepsilon^2 \bigg[ 1 - 3 \left( \alpha + 1 \right) x^2 + \left( \alpha + 1 \right) \left( \alpha + 3 \right) x^4 \\
&\qquad\qquad - \frac{\left( \alpha + 1 \right) \left( \alpha + 3 \right) \left( \alpha + 5 \right)}{15}  x^6 \bigg] \bigg\}
- \exp(3 t) \left(1 - x^2 \right)^{9 + 3 \alpha/2} + \exp(t) \left( 1 - x^2 \right)^{3 + \alpha/2}.
\end{split}
\end{equation*}
The exact solution $(\phi(x,t), \mu(x,t))$ is given by
\begin{equation*}
\phi(x,t) =
\begin{cases}
\exp(t) \left( 1 - x^2 \right)^{3 + \alpha/2}, & (x,t) \in \Omega \times (0,T], \\
0, & (x,t) \in \mathbb{R} \setminus \Omega \times (0,T],
\end{cases}
\end{equation*}
\begin{equation*}
\mu(x,t) =
\begin{cases}
\exp(t) \left( 1 - x^2 \right)^{2 + \alpha/2}, & (x,t) \in \Omega \times (0,T], \\
0, & (x,t) \in \mathbb{R} \setminus \Omega \times (0,T].
\end{cases}
\end{equation*}
%% Table 1
%\begin{table}[t]\tabcolsep=2mm
%	\caption{Numerical errors and the time convergence orders for Example 1
%		with $\sigma = 1$ and $N = 1024$.}
%	\centering
%	\begin{tabular}{cccccc}
%		\hline
%		$\alpha$ & $M$ & $Err_\infty(h,\tau)$ & $CO_{\infty, \tau}$ & $Err_2(h,\tau)$ & $CO_{2, \tau}$ \\
%		\hline
%		1.2 & 16 & 5.3533E-02 & -- & 3.7367E-02 & -- \\
%		& 32 & 1.6292E-02 & 1.7163 & 1.1293E-02 & 1.7263 \\
%		& 64 & 4.3374E-03 & 1.9093 & 3.0009E-03 & 1.9120 \\
%		& 128 & 1.1008E-03 & 1.9783 & 7.6122E-04 & 1.9790 \\
%		& 256 & 2.7080E-04 & 2.0233 & 1.8727E-04 & 2.0232 \\
%		\hline
%		1.5 & 16 & 4.1093E-02 & -- & 3.0790E-02 & -- \\
%		& 32 & 1.1674E-02 & 1.8156 & 8.7217E-03 & 1.8198 \\
%		& 64 & 3.0314E-03 & 1.9452 & 2.2631E-03 & 1.9463 \\
%		& 128 & 7.6281E-04 & 1.9906 & 5.6940E-04 & 1.9908 \\
%		& 256 & 1.8696E-04 & 2.0286 & 1.3957E-04 & 2.0285 \\
%		\hline
%		1.9 & 16 & 2.8274E-02 & -- & 2.2397E-02 & -- \\
%		& 32 & 7.4702E-03 & 1.9203 & 5.9059E-03 & 1.9231 \\
%		& 64 & 1.8951E-03 & 1.9789 & 1.4975E-03 & 1.9796 \\
%		& 128 & 4.7167E-04 & 2.0064 & 3.7271E-04 & 2.0064 \\
%		& 256 & 1.1319E-04 & 2.0590 & 8.9479E-05 & 2.0584 \\
%		\hline
%	\end{tabular}
%	\label{tab1}
%\end{table}
% Table 1
\begin{table}[t]\tabcolsep=2mm
	\caption{Numerical errors and the time convergence orders for Example 1
		with $\sigma = 1/16$ and $N = 2048$.}
	\centering
	\begin{tabular}{cccccc}
		\hline
		$\alpha$ & $M$ & $Err_\infty(h,\tau)$ & $CO_{\infty, \tau}$ & $Err_2(h,\tau)$ & $CO_{2, \tau}$ \\
		\hline
		1.2 & 8 & 1.2161E-02 & -- & 9.5625E-03 -- \\
	          & 16 & 3.5424E-03 & 1.7795 & 2.8196E-03 & 1.7619 \\
	          & 32 & 9.4599E-04 & 1.9048 & 7.5053E-04 & 1.9095 \\
	          & 64 & 2.4417E-04 & 1.9539 & 1.9338E-04 & 1.9565 \\
	          & 128 & 6.3186E-05 & 1.9502 & 4.9877E-05 & 1.9550 \\
		\hline
		1.5 & 8 & 7.3630E-03 & -- & 6.7167E-03 & -- \\
 	          & 16 & 1.8758E-03 & 1.9728 & 1.7189E-03 & 1.9663 \\
	          & 32 & 	4.7914E-04 & 1.9690 & 4.3932E-04 & 1.9681 \\
	          & 64 & 1.2258E-04 & 1.9667 & 1.1224E-04 & 1.9687 \\
	          & 128 & 3.1828E-05 & 1.9454 & 2.9011E-05 & 1.9519 \\
		\hline
		1.9 & 8 & 2.9097E-03 & -- & 2.8916E-03 & -- \\
	          & 16 & 7.4800E-04 & 1.9598 & 7.7665E-04 & 1.8965 \\
	          & 32 & 1.9194E-04 & 1.9624 & 1.9872E-04 & 1.9665 \\
	          & 64 & 4.8714E-05 & 1.9782 & 4.9951E-05 & 1.9922 \\
	          & 128 & 1.2306E-05 & 1.9850 & 1.2323E-05 & 2.0192 \\
		\hline
	\end{tabular}
	\label{tab1}
\end{table}
%% Table 2
%\begin{table}[H]\tabcolsep=2mm
%	\caption{Numerical errors and the space convergence orders for Example 1
%		with $\sigma = 1$ and $M = 1024$.}
%	\centering
%	\begin{tabular}{cccccc}
%		\hline
%		$\alpha$ & $N$ & $Err_\infty(h,\tau)$ & $CO_{\infty, h}$ & $Err_2(h,\tau)$ & $CO_{2, h}$ \\
%		\hline
%		1.2 & 16 & 4.2457E-02 & -- & 2.9776E-02 & -- \\
%		& 32 & 1.0011E-02 & 2.0844 & 6.8793E-03 & 2.1138 \\
%		& 64 & 2.3681E-03 & 2.0798 & 1.6310E-03 & 2.0765 \\
%		& 128 & 5.6039E-04 & 2.0792 & 3.8470E-04 & 2.0840 \\
%		& 256 & 1.2321E-04 & 2.1853 & 8.4649E-05 & 2.1842 \\
%		\hline
%		1.5 & 16 & 3.8190E-02 & -- & 2.8015E-02 & -- \\
%		& 32 & 8.4954E-03 & 2.1684 & 6.2749E-03 & 2.1585 \\
%		& 64 & 1.9436E-03 & 2.1280 & 1.4399E-03 & 2.1236 \\
%		& 128 & 4.4288E-04 & 2.1337 & 3.2854E-04 & 2.1318 \\
%		& 256 & 9.4942E-05 & 2.2218 & 7.0452E-05 & 2.2214 \\
%		\hline
%		1.9 & 16 & 3.4622E-02 & -- & 2.8402E-02 & -- \\
%		& 32 & 8.3803E-03 & 2.0466 & 6.5507E-03 & 2.1163 \\
%		& 64 & 1.9932E-03 & 2.0719 & 1.5598E-03 & 2.0703 \\
%		& 128 & 4.7324E-04 & 2.0744 & 3.7047E-04 & 2.0739 \\
%		& 256 & 1.0819E-04 & 2.1290 & 8.4684E-05 & 2.1292 \\
%		\hline
%	\end{tabular}
%	\label{tab2}
%\end{table}
% Table 2
\begin{table}[H]\tabcolsep=2mm
	\caption{Numerical errors and the space convergence orders for Example 1
		with $\sigma = 1/16$ and $M = 1024$.}
	\centering
	\begin{tabular}{cccccc}
		\hline
		$\alpha$ & $N$ & $Err_\infty(h,\tau)$ & $CO_{\infty, h}$ & $Err_2(h,\tau)$ & $CO_{2, h}$ \\
		\hline
        1.2 & 16 & 4.2478E-02 & -- & 2.9791E-02 & -- \\
	          & 32 & 1.0030E-02 & 2.0824 & 6.8923E-03 & 2.1118 \\
	          & 64 & 2.3864E-03 & 2.0714 & 1.6438E-03 & 2.0680 \\
	          & 128 & 5.7878E-04 & 2.0437 & 3.9746E-04 & 2.0482 \\
	          & 256 & 1.4162E-04 & 2.0310 & 9.7434E-05 & 2.0283 \\
		\hline
		1.5 & 16 & 3.8204E-02 & -- & 2.8025E-02 & -- \\
	          & 32 & 	8.5079E-03 & 2.1668 & 6.2842E-03 & 2.1569 \\
	          & 64 & 1.9561E-03 & 2.1208 & 1.4492E-03 & 2.1165 \\
	          & 128 & 4.5532E-04 & 2.1030 & 3.3786E-04 & 2.1008 \\
	          & 256 & 1.0742E-04 & 2.0836 & 7.9805E-05 & 2.0819 \\
		\hline
		1.9 & 16 & 3.4629E-02 & -- & 2.8408E-02 & -- \\
	          & 32 & 8.3878E-03 & 2.0456 & 6.5566E-03 & 2.1153 \\
	          & 64 & 2.0006E-03 & 2.0679 & 1.5657E-03 & 2.0661 \\
	          & 128 & 4.8060E-04 & 2.0575 & 3.7631E-04 & 2.0568 \\
	          & 256 & 1.1556E-04 & 2.0562 & 9.0526E-05 & 2.0555 \\
		\hline
	\end{tabular}
	\label{tab2}
\end{table}
% Figure 1
\begin{figure}[H]
	\centering
	\includegraphics[width=3.0in,height=3.0in]{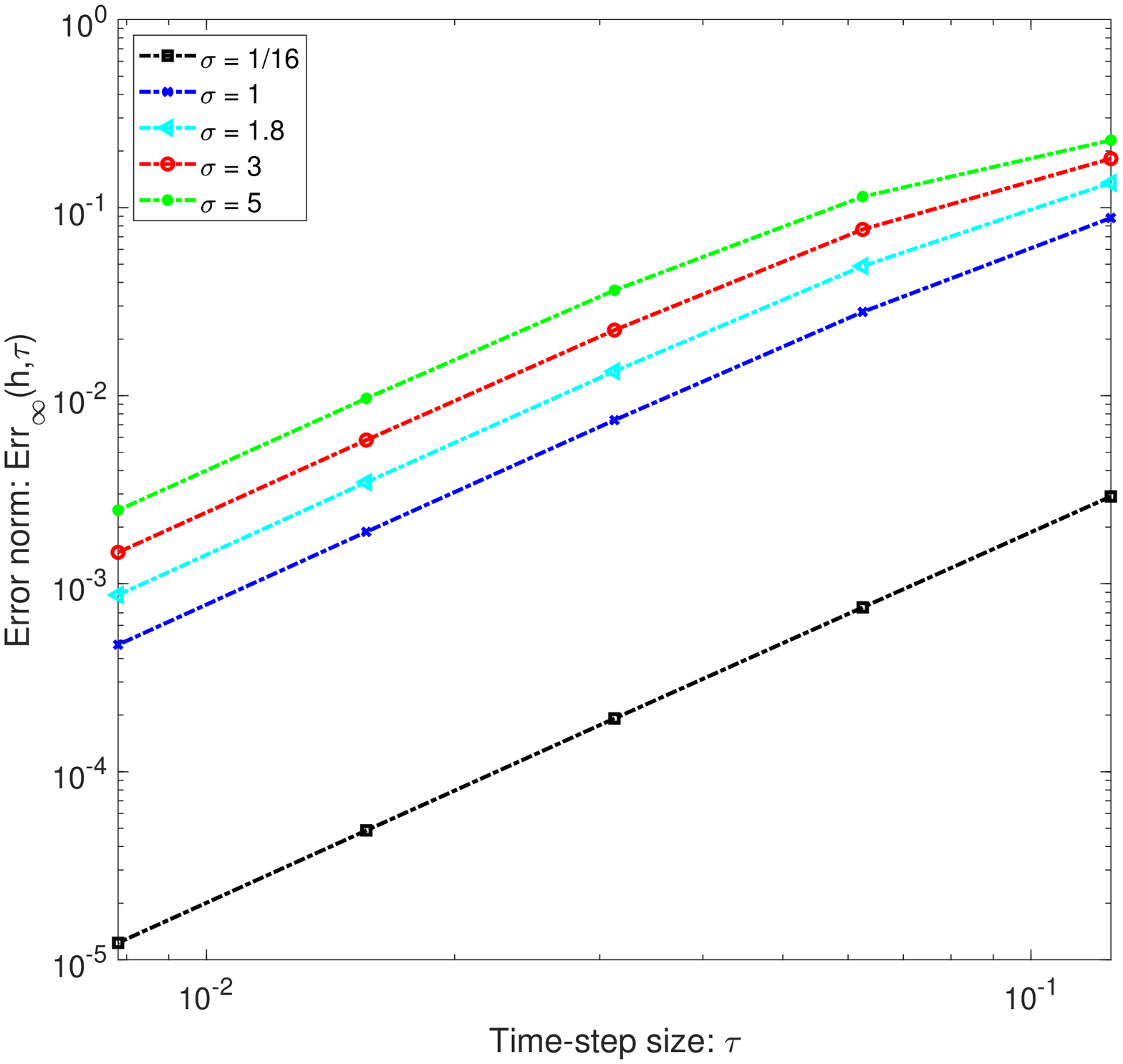}
	\includegraphics[width=3.0in,height=3.0in]{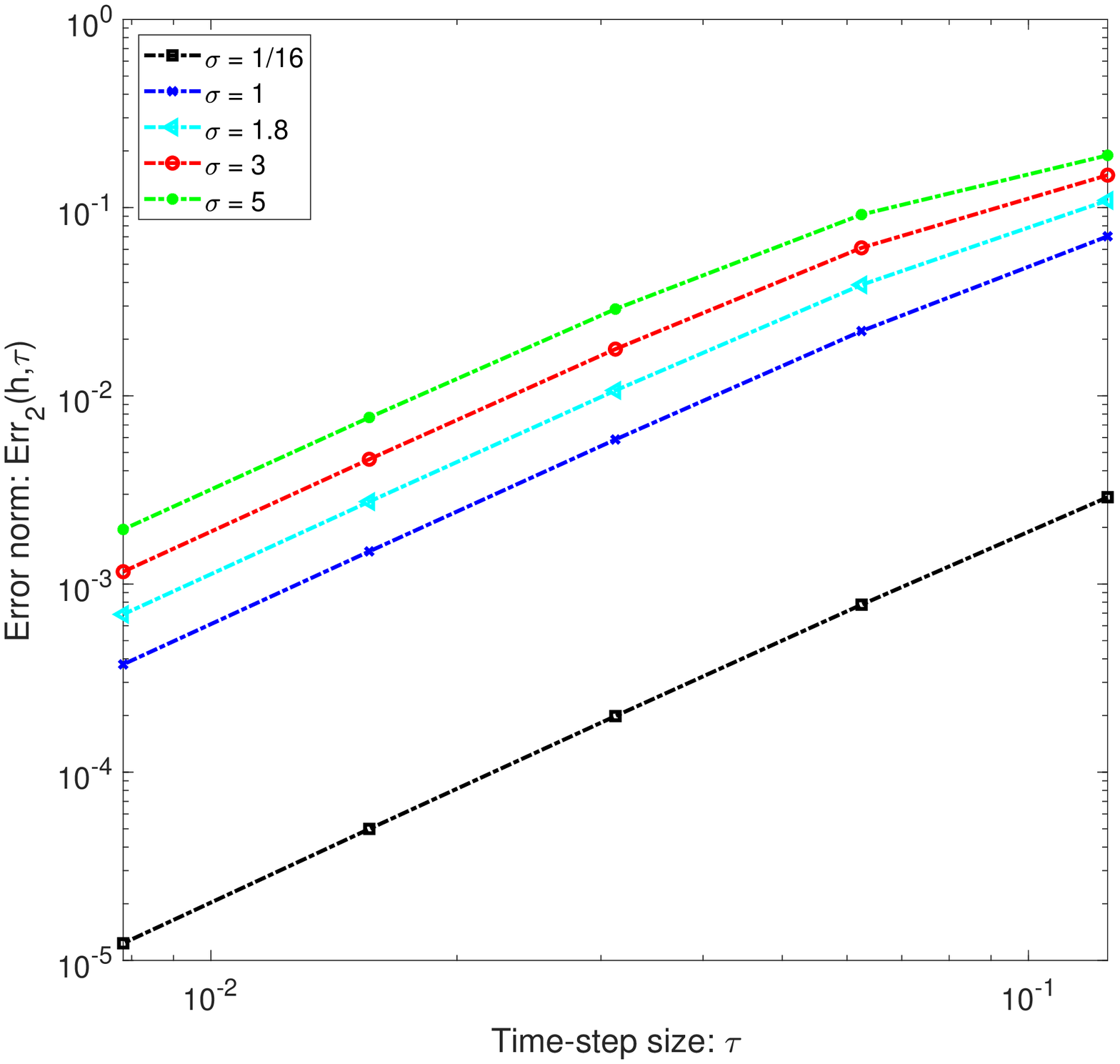}
	\caption{The numerical errors for Example 1 with $\alpha=1.9$, $h=1/1024$ and different values of $\sigma$ and $\tau$.}
	\label{fig1}
\end{figure}

The errors in Table \ref{tab1} decrease steadily with increasing $M$ the number of time steps,
and the convergence order in time is $2$ as expected.
According to the work \cite{duo2018novel}, since $\left( 1 - x^2 \right)^{3 + \alpha/2} \in C^{3, \alpha/2} \left( \Omega \right)$
(see the Appendix for a lower space regularity case),
the spatial convergence order is also $2$ (i.e., $p = 2$). Table \ref{tab2} reports the errors and the convergence order in space.
It shows that for fixed $M = 1024$, the convergence order in space is indeed $2$.
This is in good agreement with our theoretical analysis in Section \ref{sec3.2}.

On the other hand, it should be discussed why using $\sigma = 1/16$ in our paper.
From Section \ref{sec3.2}, we know that if $\sigma \geq 1/16$, our scheme \eqref{eq3.2} is energy stable and convergence.
However, this does not mean that every $\sigma \geq 1/16$ can get small numerical errors and a `good' temporal convergence order.
Fig.\ref{fig1} is plotted to discuss the effect of different values of $\sigma$ ($\sigma=1/16,1,1.8,3,5$)
on the temporal convergence order of \eqref{eq3.2},
where $\alpha=1.9$, $h=1/1024$ and $\tau=1/8,1/16,1/32,1/64,1/128$.
It can be observed that the mBDF2 scheme \eqref{eq3.2} with $\sigma=1/16$ yields the smallest numerical errors
and a `good' slope. Thus, we use $\sigma=1/16$ in this article.
%Moreover, fix $A = 1$ in our mBDF2 scheme \eqref{eq3.2}.

\subsection{Fast implementation}
\label{sec5.2}

In this subsection, the performance of our preconditioning technique
(i.e., Algorithm \ref{alg1} with preconditioner $P_{L}^{j + 1 (\ell)}$) is reported
and compared with the preconditioner $\hat{P}_{L}^{j + 1 (\ell)}$.
% Figure 2
\begin{figure}[H]
	\centering
	\subfigure[$\phi(x,T)$]
	{\includegraphics[width=3.0in,height=2.2in]{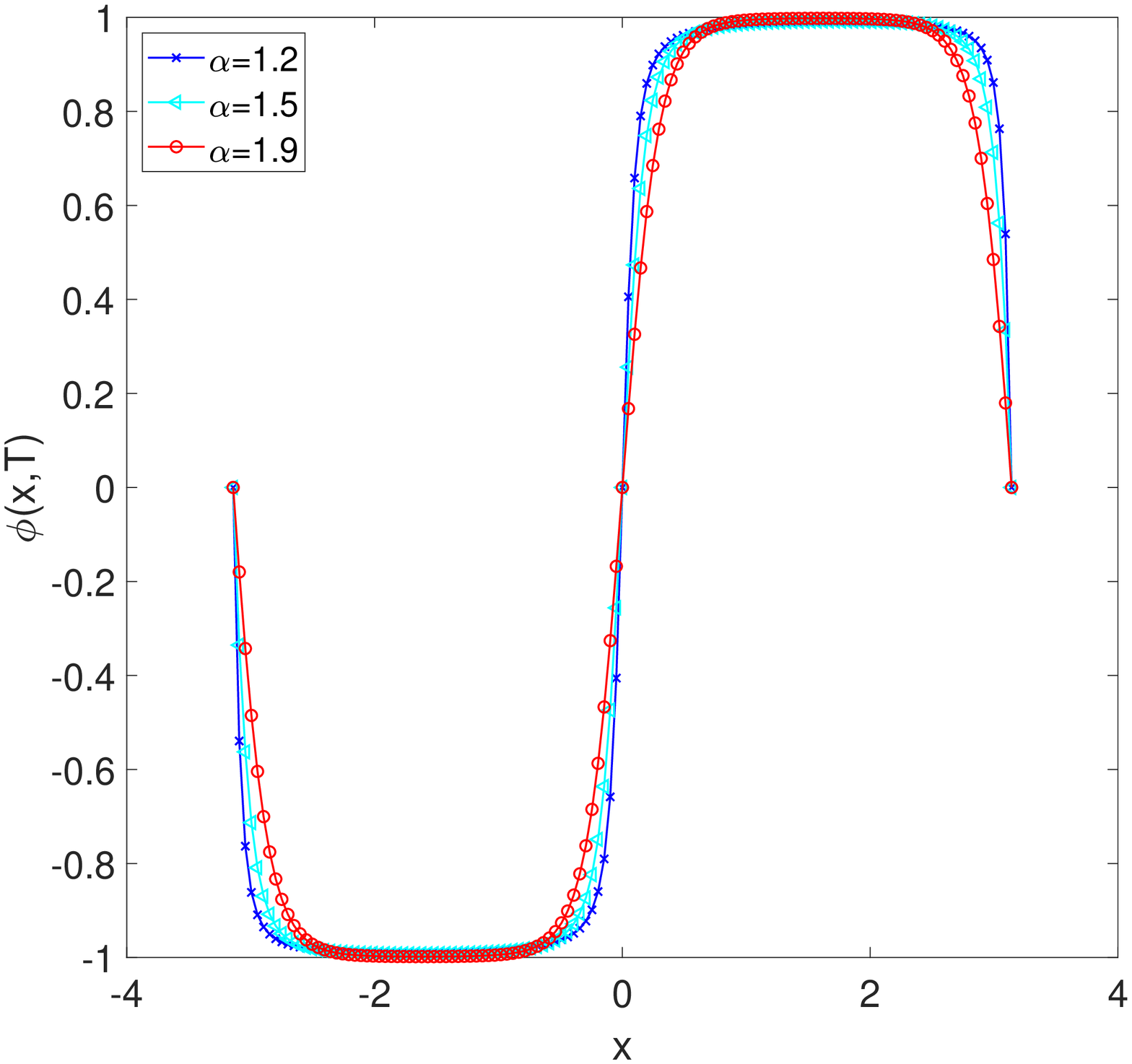}}
	\subfigure[$\alpha = 1.2$]
	{\includegraphics[width=3.0in,height=2.25in]{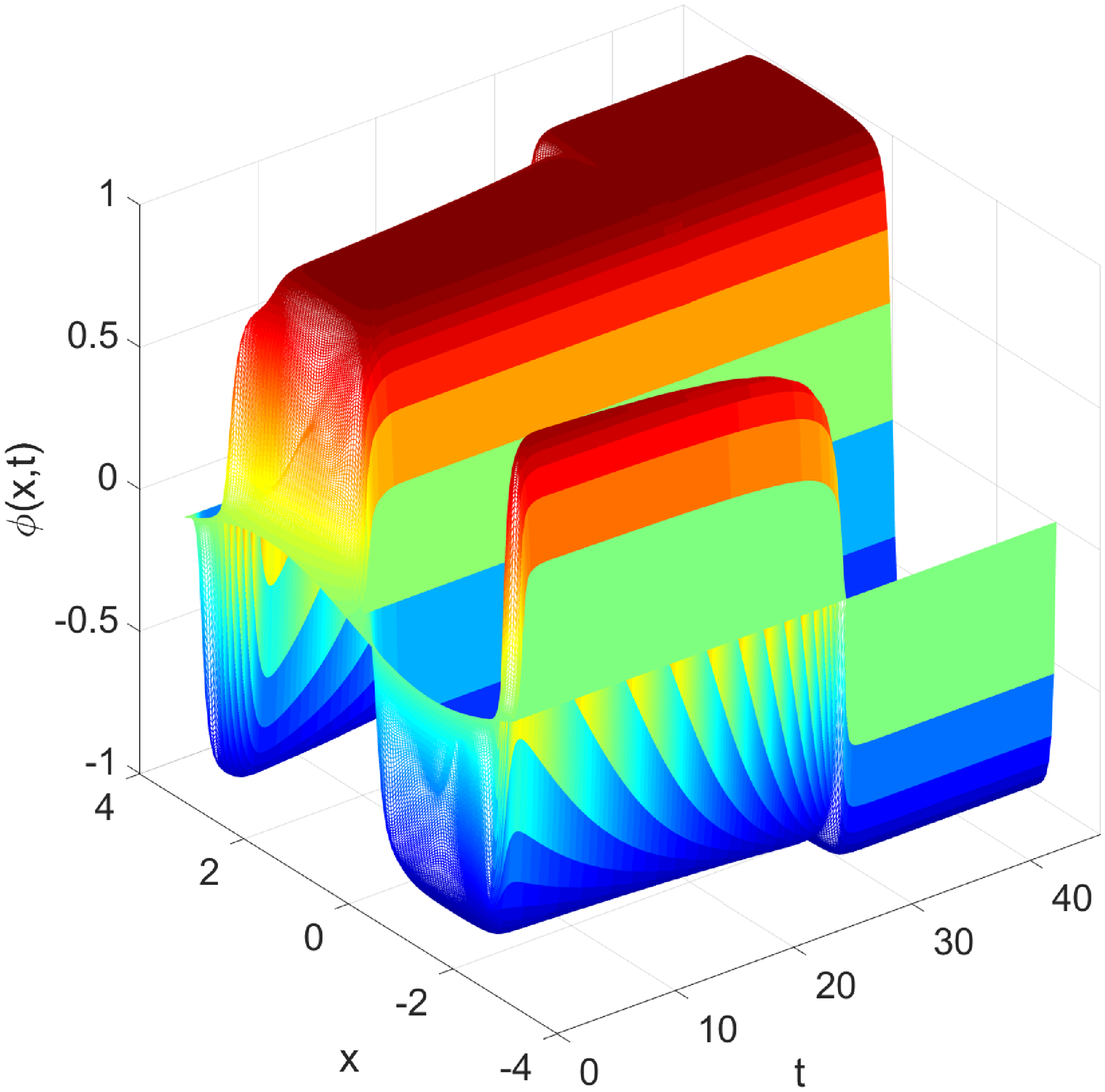}}\\
	\subfigure[$\alpha = 1.5$]
	{\includegraphics[width=3.0in,height=2.2in]{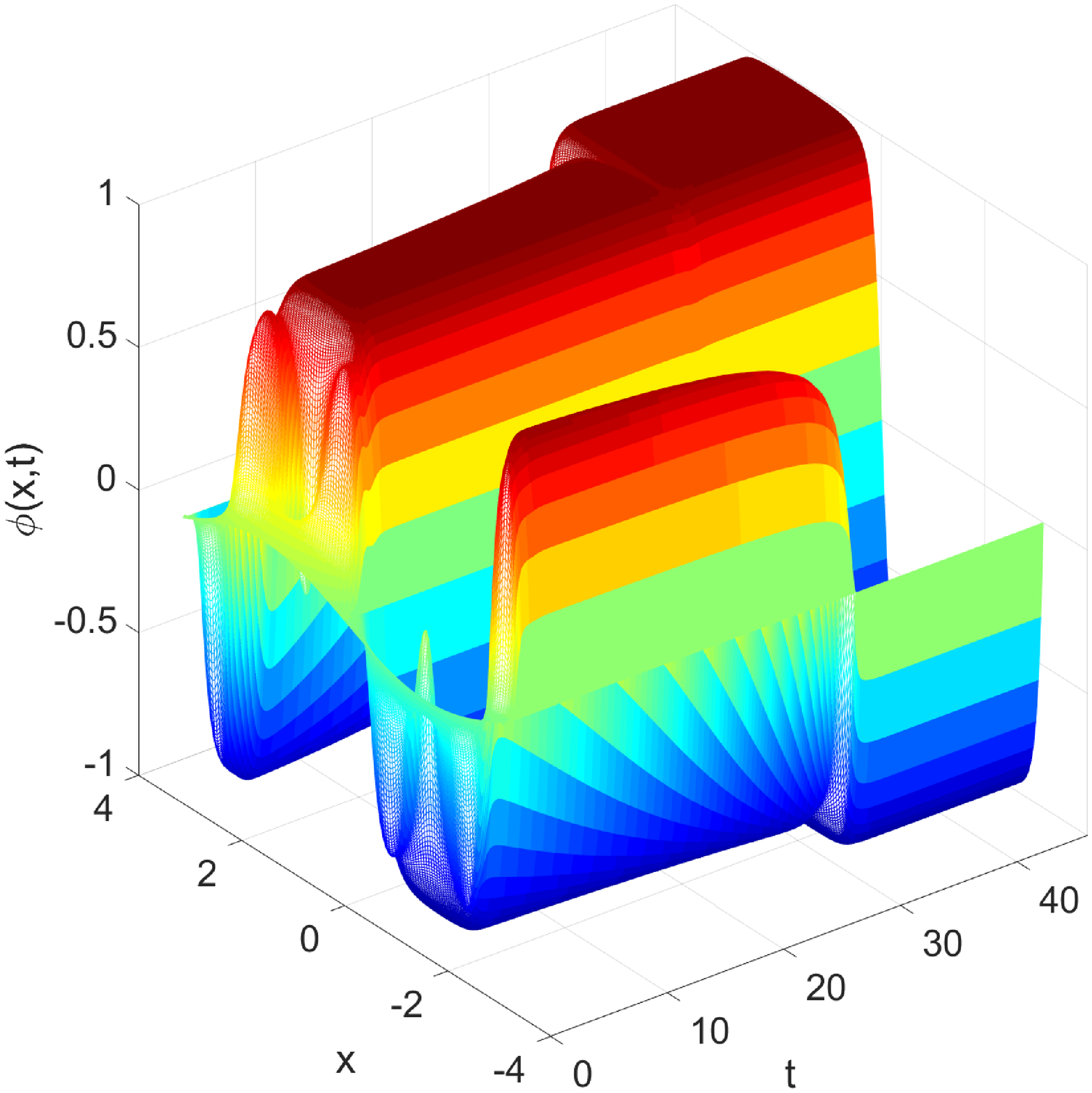}}
	\subfigure[$\alpha = 1.9$]
	{\includegraphics[width=3.0in,height=2.2in]{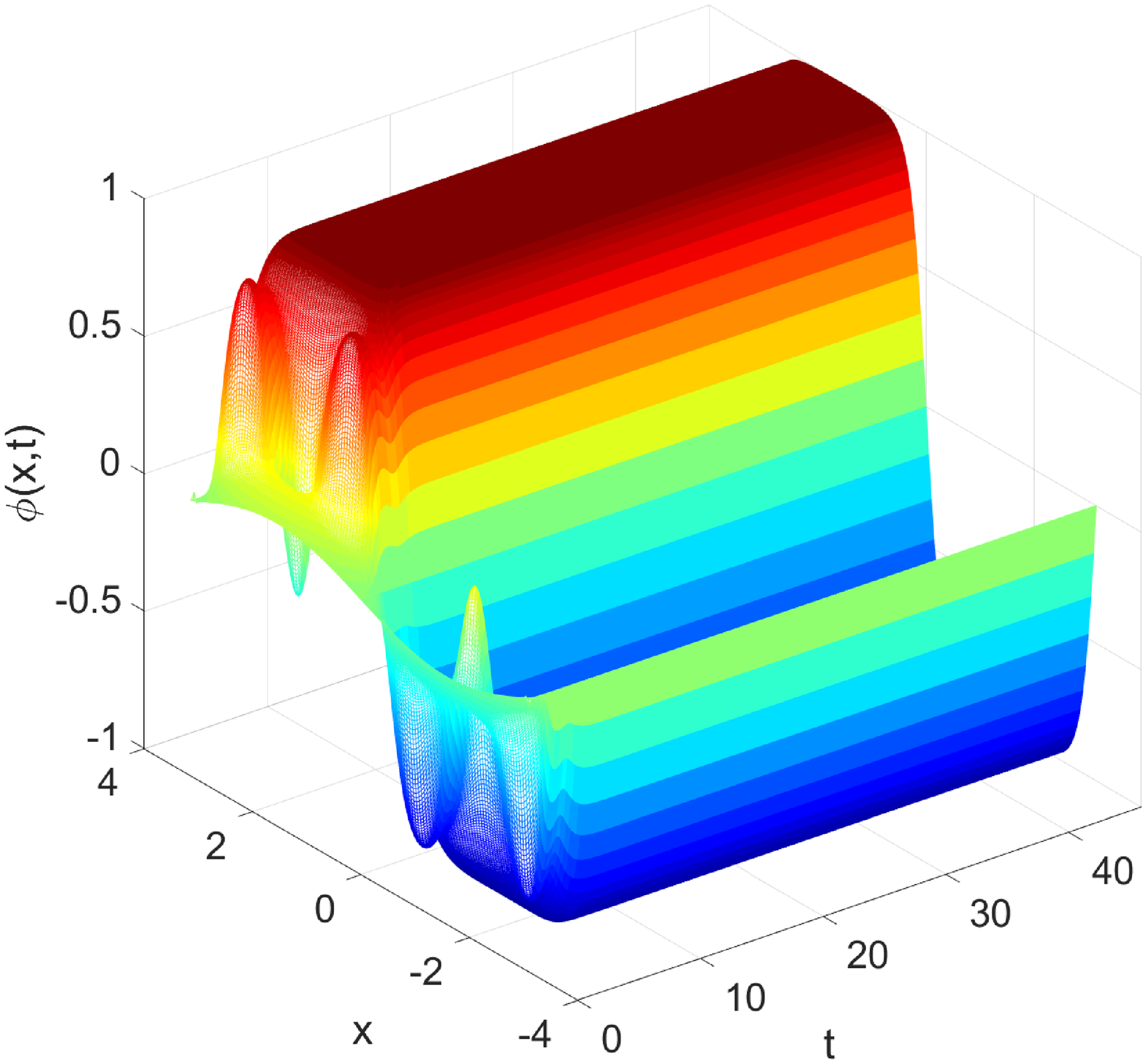}}
	\caption{Comparison of the numerical solutions with $(M, N) = (2048, 128)$ and different $\alpha$ for Example 2.}
	\label{fig2}
\end{figure}

\noindent \textbf{Example 2.} Consider Eq.~\eqref{eq1.1} with $L = \pi$, $\varepsilon^2 = 0.05$
and initial value $ \phi_0(x) = 0.1 \sin(x)$.

Table \ref{tab3} shows that compared with the direct BS method, the two preconditioned iterative methods (i.e., $\mathcal{P}$ and $\mathcal{P}_s$)
greatly reduce the computational cost. This results in a much smaller CPU time.
For small $\alpha = 1.2, 1.5$, the numbers $\mathrm{Iter2}$ of the methods $\mathcal{P}$ and $\mathcal{P}_s$ are not strongly influenced by the mesh size.
Comparing with the method $\mathcal{P}_s$, the number $\mathrm{Iter2}$ of $\mathcal{P}$ is slightly smaller.
However, the numbers $\mathrm{Iter2}$ of both preconditioned iterative methods are not very satisfactory for $\alpha = 1.9$.
Thus, a more efficient preconditioner for solving \eqref{eq4.1} should be considered in our future work.
For test problems with $\alpha = 1.5, 1.9$, the CPU time of method $\mathcal{P}$ is smaller than that of method $\mathcal{P}_s$.
On the other hand, when $\alpha = 1.2$, the CPU time of $\mathcal{P}_s$ is smaller than that of $\mathcal{P}$.
Notice that both methods have almost the same value of $\mathrm{Iter2}$. Thus,
the reason may be that method $\mathcal{P}$ needs more arithmetics with complex numbers.
\begin{table}[t]\footnotesize\tabcolsep=2.0pt
	\begin{center}
		\caption{Results of different methods when $T = 46$ and $M = N$ for Example 2.}
		\centering
		\begin{tabular}{cccccccc}
			\hline
			& & \multicolumn{2}{c}{\rm{BS}} & \multicolumn{2}{c}{$\mathcal{P}_s$}
			& \multicolumn{2}{c}{$\mathcal{P}$} \\
			[-2pt] \cmidrule(lr){3-4} \cmidrule(lr){5-6} \cmidrule(lr){7-8}\\ [-11pt]
			$\alpha$ & $N$ & $\mathrm{Iter1}$ & $\mathrm{Time}$
			&($\mathrm{Iter1}$, $\mathrm{Iter2}$) & $\mathrm{Time}$
			& ($\mathrm{Iter1}$, $\mathrm{Iter2}$) & $\mathrm{Time}$ \\
			\hline
			1.2 & 64 & 3.4 & 0.085 & (3.4, 14.2) & 0.195 & 	(3.4, 13.9) & 0.199 \\
	              & 128 & 3.2 & 0.741 & 	(3.2, 16.1) & 0.669 & (3.2, 15.5) & 0.644 \\
	              & 256 & 2.8 & 4.943 & 	(2.8, 16.6) & 1.778 & (2.8, 16.0) & 1.826 \\
	              & 512 & 2.7 & 49.734 & (2.7, 16.1) & 8.093 & 	(2.7, 15.8) & 8.489 \\
	              & 1024 & 2.7 & 486.649 & (2.7, 15.6) & 19.975 & (2.7, 15.6) & 20.636 \\
	              & 2048 & 2.6 & 5392.398 & (2.6, 15.2) & 108.659 & (2.6, 15.3) & 116.275 \\
			\\
			1.5 & 64 & 2.6 & 0.064 & (2.6, 14.4) & 0.146 & 	(2.6, 14.1) & 0.148 \\
	              & 128 & 2.6 & 0.595 & 	(2.6, 16.3) & 0.553 & (2.6, 16.0) & 0.537 \\
	              & 256 & 2.6 & 4.592 & 	(2.6, 16.8) & 1.704 & (2.6, 16.5) & 1.692 \\
	              & 512 & 2.7 & 48.418 & (2.7, 16.9) & 8.410 & 	(2.7, 16.6) & 8.858 \\
	              & 1024 & 2.7 & 493.932 & (2.7, 17.4) & 22.012 & (2.7, 16.3) & 21.640 \\
	              & 2048 &	2.5 & 5236.563 & (2.5, 19.3) & 130.851 & (2.5, 17.1) & 124.670 \\
			\\
			1.9 & 64 & 2.5 & 0.061 & (2.5, 15.2) & 0.152 & 	(2.5, 14.9) & 0.150 \\
	               & 128 & 1.8 & 0.386 & (1.8, 16.5) & 0.377 & (1.8, 16.3) & 0.360 \\
	               & 256 & 1.6 & 2.794 & (1.6, 18.7) & 1.171 & (1.6, 17.0) & 1.081 \\
	               & 512 & 1.6 & 28.695 & (1.6, 25.4) & 7.832 & (1.6, 21.3) & 6.576 \\
	               & 1024 & 1.5 & 	277.725 & 	(1.5, 33.8) & 26.981 & (1.5, 27.5) & 20.660 \\
	               & 2048 & 1.5 & 3099.148 & (1.5, 172.8) & 1053.756 & (1.5, 169.0) & 1042.039 \\
			\hline
		\end{tabular}
		\label{tab3}
	\end{center}
\end{table}
% Figure 3
\begin{figure}[H]
	\centering
	\includegraphics[width=3.5in,height=3.0in]{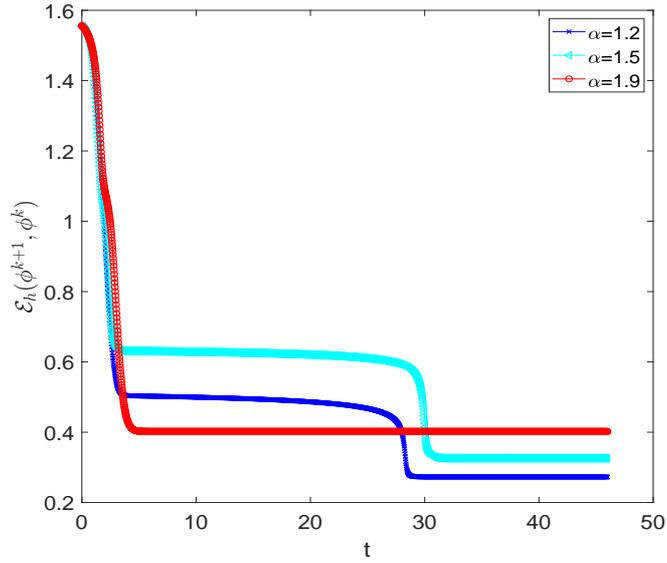}
	\caption{Energy dissipation with $(M, N) = (2048, 128)$ and $T = 46$ for Example 2 for different values of $\alpha$.}
	\label{fig3}
\end{figure}
% Figure 4
\begin{figure}[H]
	\centering
	\subfigure[Eigenvalues of $\mathcal{J}^{2(0)}$]
	{\includegraphics[width=3.0in,height=2.5in]{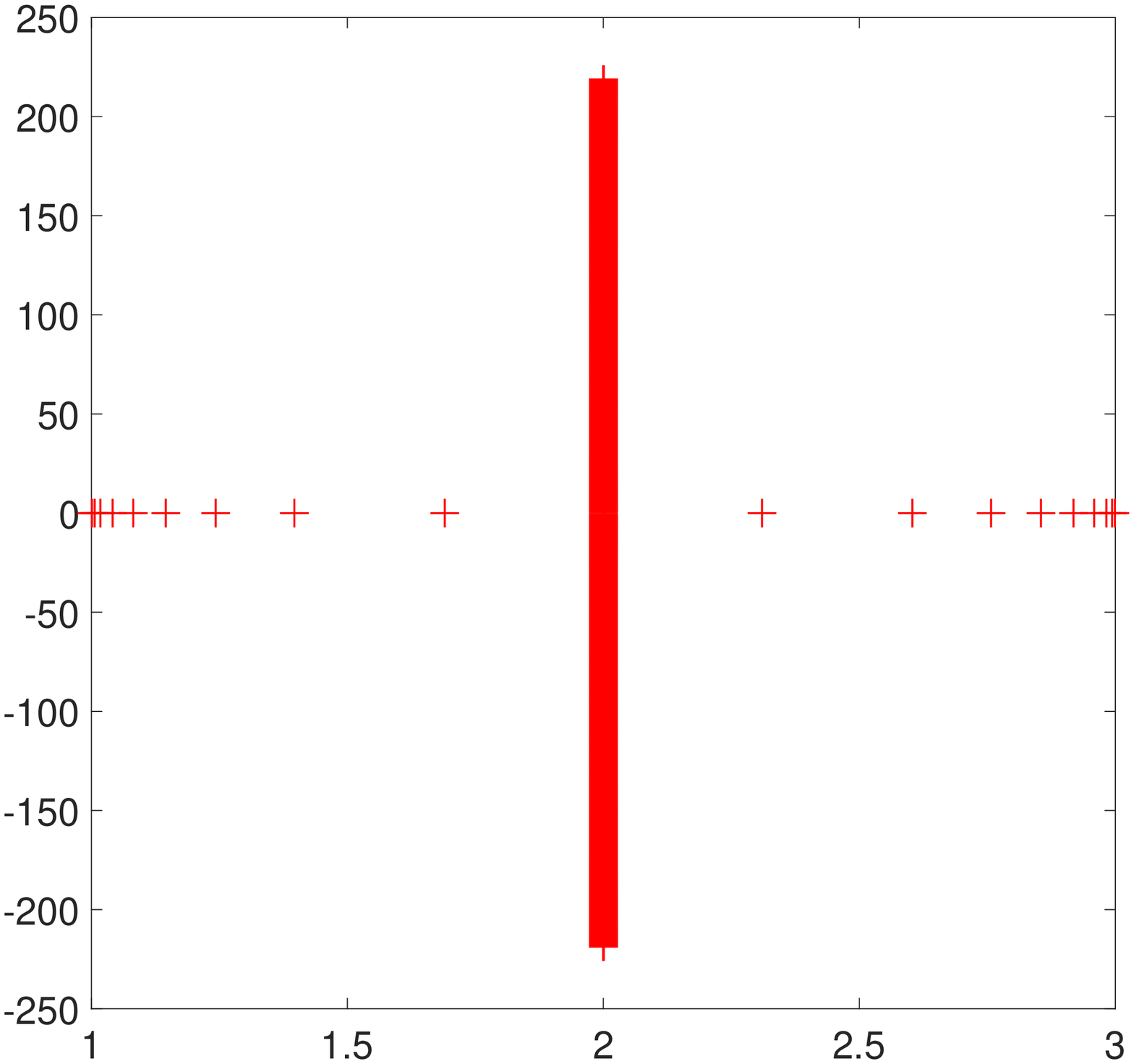}}
	\subfigure[Eigenvalues of $\left( P_{L}^{2(0)} \right)^{-1} \mathcal{J}^{2(0)}$ ({\color{blue} $*$})
	and $\left( \hat{P}_{L}^{2(0)} \right)^{-1} \mathcal{J}^{2(0)}$ ({\color{magenta} $\circ$})]
	{\includegraphics[width=3.0in,height=2.56in]{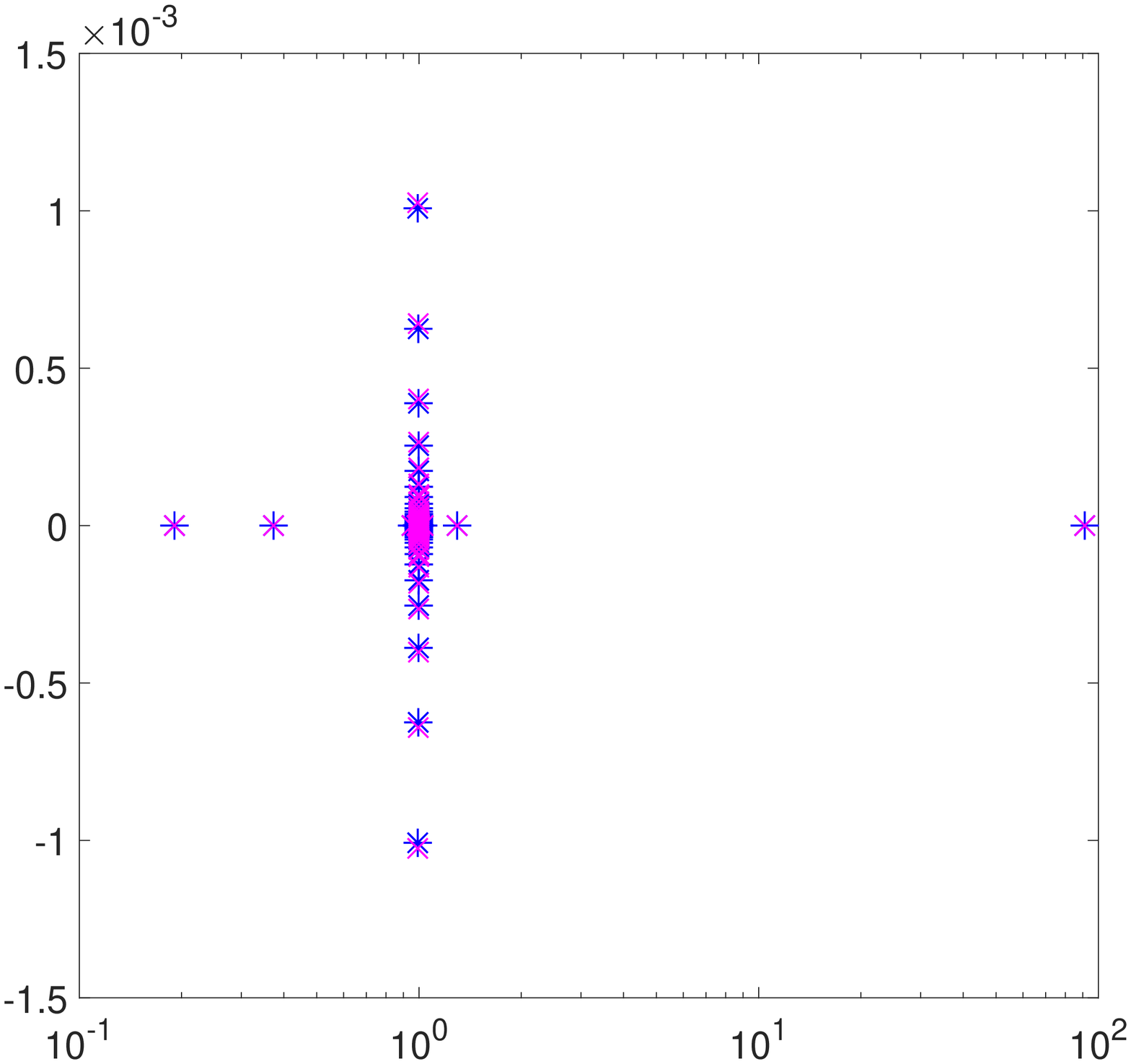}}\\
	\subfigure[Eigenvalues of $G$]
	{\includegraphics[width=3.0in,height=2.5in]{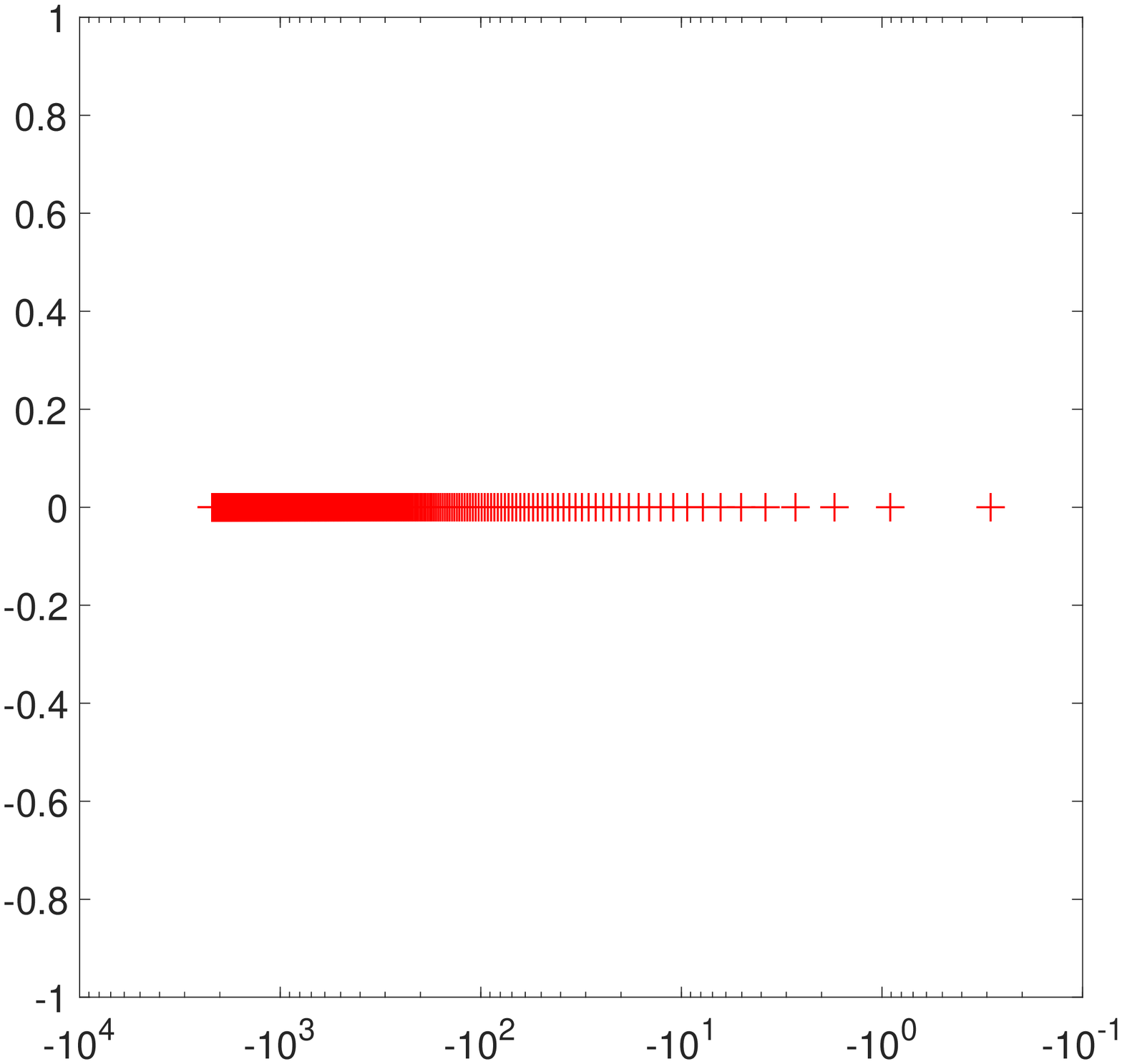}}
	\subfigure[Eigenvalues of $\left( sk(G) \right)^{-1} G$ ({\color{blue} $*$})
	and $\left( s(G) \right)^{-1} G$ ({\color{magenta} $\circ$})]
	{\includegraphics[width=3.0in,height=2.5in]{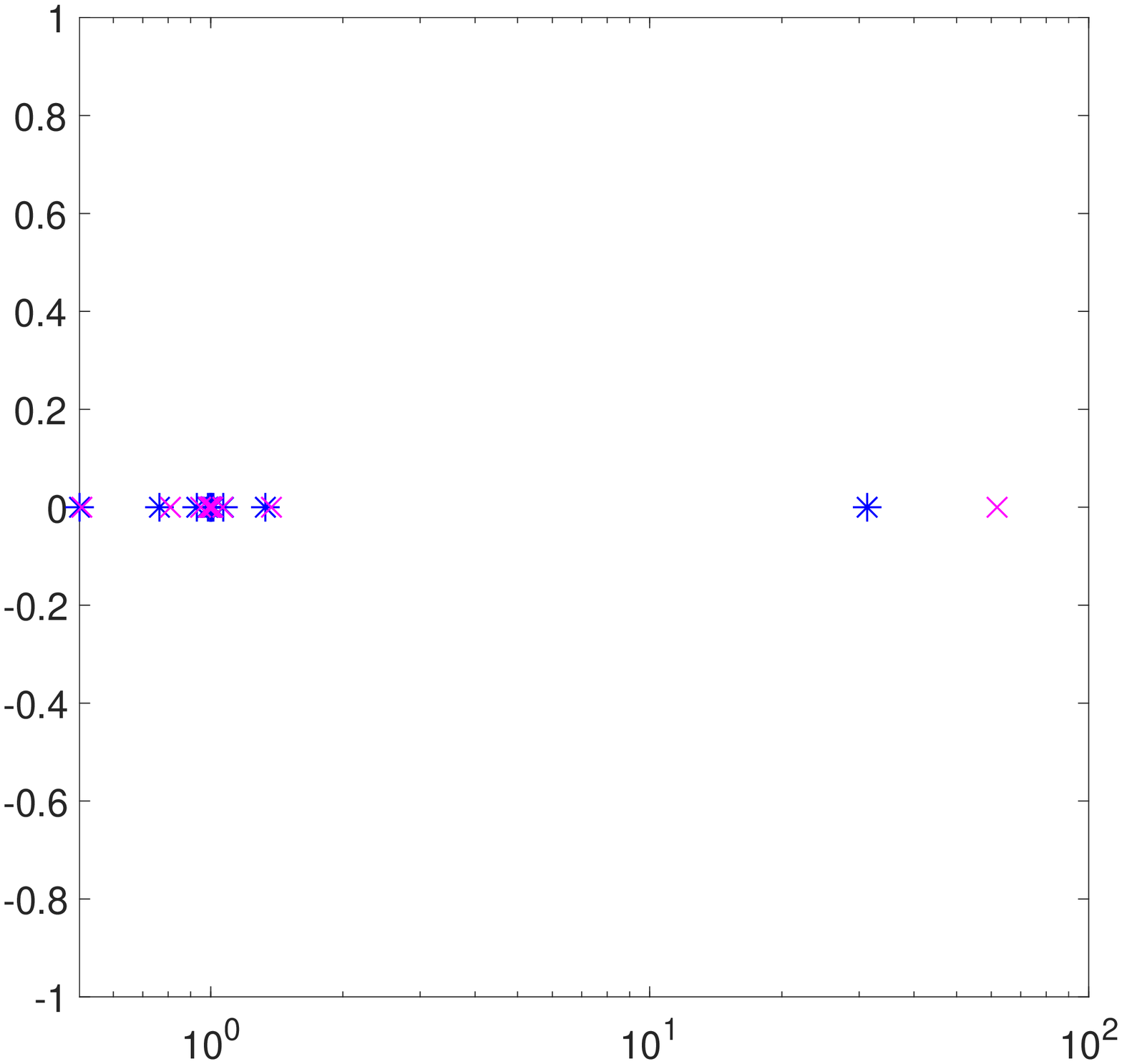}}
	\caption{Spectra of $\mathcal{J}^{2(0)}$, $\left( P_{L}^{2(0)} \right)^{-1} \mathcal{J}^{2(0)}$,
		$\left( \hat{P}_{L}^{2(0)} \right)^{-1} \mathcal{J}^{2(0)}$, $G$,
		$\left( sk(G) \right)^{-1} G$ and $\left( s(G) \right)^{-1} G$ for $\alpha = 1.5$ and $M = N = 512$ in Example 2.}
	\label{fig4}
\end{figure}

Fig.~\ref{fig2}(a) shows the numerical solutions at $T = 46$ for different values of $\alpha$,
and the increasing effect of the interface thickness as $\alpha$ increases.
The time evolution of the model \eqref{eq1.1} for varying $\alpha$ is shown in Figs.~\ref{fig2}(b-d).
It is clearly visible from it that the larger $\alpha$, the shorter the lifetime of the unstable interface,
eventually becoming fully stabilized due to the long-range dependence and the heavy-tailed influence of the fractional diffusion process.
Fig.~\ref{fig3} shows how the energy of SFCH \eqref{eq1.1} is dissipated. This also confirms the analytical result in Theorem \ref{th3.1}.
The spectra of $\mathcal{J}^{2(0)}$, $\left( P_{L}^{2(0)} \right)^{-1} \mathcal{J}^{2(0)}$,
$\left( \hat{P}_{L}^{2(0)} \right)^{-1} \mathcal{J}^{2(0)}$, $G$,
$\left( sk(G) \right)^{-1} G$ and $\left( s(G) \right)^{-1} G$ are drawn in Fig.~\ref{fig4}.
In Fig.~\ref{fig4}(b), the eigenvalues of $\left( P_{L}^{2(0)} \right)^{-1} \mathcal{J}^{2(0)}$
are slightly more clustered than
$\left( \hat{P}_{L}^{2(0)} \right)^{-1} \mathcal{J}^{2(0)}$. In Fig.~\ref{fig4}(d),
all eigenvalues of $\left( sk(G) \right)^{-1} G$ and $\left( s(G) \right)^{-1} G$ are clustered around $1$, except for several outliers.
In a word, Table \ref{tab3} and Fig.~\ref{fig4} indicate that our preconditioner $P_{L}^{j + 1(\ell)}$ performs slightly better than
$\hat{P}_{L}^{j + 1(\ell)}$.

\section{Concluding remarks}
\label{sec6}

In this article, we propose and analyse an energy stable finite difference nonlinear scheme \eqref{eq3.2} with second-order accuracy in time
to approximate the SFCH \eqref{eq1.1}. For the temporal discretization, the BDF2 scheme combined
with a second-order extrapolation formula applied to the concave term is applied.
However, the resulting scheme is not energy stable. To save the energy stability of the numerical scheme, an artificial
second-order Douglas-Dupont regularization is considered resulting in the scheme mBDF2. The energy stability and convergence of the mBDF2 scheme \eqref{eq3.2}
are proved. To reduce the computational cost, a preconditioning technique is designed to accelerate the solution of the involved nonlinear system.
Numerical examples strongly support the efficiency of our preconditioning technique.
For large $\alpha$ (i.e., $\alpha = 1.9$), the behaviour of our preconditioner is less satisfactory.
Another more efficient preconditioners should be considered in our future work,
such as the Sine transform-based preconditioner \cite{ng2004iterative} and the two-stage preconditioner \cite{roy2019block}.
Furthermore, our method can easily be extended to the variable time-step version \cite{chen2019second} and to the higher dimensional case.

\section*{Acknowledgments}
\addcontentsline{toc}{section}{Acknowledgments}
\label{sec7}

\textit{This research is supported by the National Natural Science Foundation of China
(Nos.~61876203, 61772003, 11801463 and 11701522) and the Fundamental Research Funds for
the Central Universities (No.~JBK1902028). The first author is also supported by the China Scholarship Council.}

\section*{Appendix}
\addcontentsline{toc}{section}{Appendix}
\label{sec8}
In this appendix, the convergence orders of \eqref{eq3.2} with lower spatial regularity are provided.

\noindent\textbf{Example A.} Consider Example 1 with the exact solution and source terms given by
\begin{equation*}
\phi(x,t) = \mu(x,t) =
\begin{cases}
\exp(t) \left( 1 - x^2 \right)^{1 + \alpha/2}, & (x,t) \in \Omega \times (0,T], \\
0, & (x,t) \in \mathbb{R} \setminus \Omega \times (0,T],
\end{cases}
\end{equation*}
\begin{equation*}
f(x,t) = \exp(t) \left\{ \left( 1 - x^2 \right)^{1 + \alpha/2}
+ \frac{2^{\alpha} \Gamma(\frac{\alpha + 1}{2}) \Gamma(2 + \alpha/2)}{\sqrt{\pi} \Gamma(2)}
\left[ 1 - \left( \alpha + 1 \right) x^2 \right] \right\}
\end{equation*}
and
\begin{equation*}
\begin{split}
& \psi(x,t) = \exp(t) \left\{ \left( 1 - x^2 \right)^{1 + \alpha/2}
- \frac{2^{\alpha} \Gamma(\frac{\alpha + 1}{2}) \Gamma(2 + \alpha/2)}{\sqrt{\pi} \Gamma(2)}
\varepsilon^2 \left[ 1 - \left( \alpha + 1 \right) x^2 \right] \right\} \\
&\qquad\qquad - \exp(3 t) \left(1 - x^2 \right)^{3 + 3 \alpha/2} + \exp(t) \left( 1 - x^2 \right)^{1 + \alpha/2},
\end{split}
\end{equation*}
respectively.
%% Table A1
%\begin{table}[t]\tabcolsep=2mm
%	\caption{Numerical errors and the time convergence orders for Example A
%		with $\sigma = 1$ and $N = 2048$.}
%	\centering
%	\begin{tabular}{cccccc}
%		\hline
%		$\alpha$ & $M$ & $Err_\infty(h,\tau)$ & $CO_{\infty, \tau}$ & $Err_2(h,\tau)$ & $CO_{2, \tau}$ \\
%		\hline
%		1.2 & 8 & 5.9259E-02 & -- & 3.3816E-02 & -- \\
%		& 16 & 1.9640E-02 & 1.5932 & 1.1047E-02 & 1.6141 \\
%		& 32 & 5.3134E-03 & 1.8861 & 2.9805E-03 & 1.8900 \\
%		& 64 & 1.3602E-03 & 1.9658 & 7.6250E-04 & 1.9667 \\
%		& 128 & 3.4342E-04 & 1.9858 & 1.9251E-04 & 1.9858 \\
%		\hline
%		1.5 & 8 & 4.3052E-02 & -- & 2.7281E-02 & -- \\
%		& 16 & 1.2646E-02 & 1.7674 & 7.8960E-03 & 1.7887 \\
%		& 32 & 3.3047E-03 & 1.9361 & 2.0551E-03 & 1.9419 \\
%		& 64 & 8.3889E-04 & 1.9780 & 5.2114E-04 & 1.9795 \\
%		& 128 & 2.1163E-04 & 1.9869 & 1.3146E-04 & 1.9870 \\
%		\hline
%		1.9 & 8 & 3.4422E-02 & -- & 2.4446E-02 & -- \\
%		& 16 & 9.8519E-03 & 1.8049 & 6.8891E-03 & 1.8272 \\
%		& 32 & 2.5654E-03 & 1.9412 & 1.7866E-03 & 1.9471 \\
%		& 64 & 6.5070E-04 & 1.9791 & 4.5269E-04 & 1.9806 \\
%		& 128 & 1.6410E-04 & 1.9874 & 1.1413E-04 & 1.9878 \\
%		\hline
%	\end{tabular}
%	\label{tabX1}
%\end{table}
% Table A1
\begin{table}[t]\tabcolsep=2mm
	\caption{Numerical errors and the time convergence orders for Example A
		with $\sigma = 1/16$ and $N = 2048$.}
	\centering
	\begin{tabular}{cccccc}
		\hline
		$\alpha$ & $M$ & $Err_\infty(h,\tau)$ & $CO_{\infty, \tau}$ & $Err_2(h,\tau)$ & $CO_{2, \tau}$ \\
		\hline
		1.2 & 16 & 1.6538E-02 & -- & 9.1950E-03 & -- \\
 	           & 32 & 4.5744E-03 & 1.8541 & 	2.5365E-03 & 1.8580 \\
	           & 64 & 1.1803E-03 & 1.9544 & 	6.5411E-04 & 1.9552 \\
	           & 128 & 2.9897E-04 & 1.9811 & 1.6571E-04 & 1.9809 \\
	           & 256 & 7.5839E-05 & 1.9790 & 4.2069E-05 & 1.9778 \\
		\hline
		1.5 & 16 & 1.2238E-02 & -- & 7.6299E-03 & -- \\
	          & 32 & 3.2328E-03 & 1.9205 & 2.0076E-03 & 1.9262 \\
	          & 64 & 8.2255E-04 & 1.9746 & 5.1029E-04 & 1.9761 \\
	          & 128 & 2.0766E-04 & 1.9859 & 1.2882E-04 & 1.9860 \\
	          & 256 & 5.2845E-05 & 1.9744 & 3.2809E-05 & 1.9732 \\
		\hline
		1.9 & 16 & 9.5353E-03 & -- & 6.6851E-03 & -- \\
	          & 32 & 2.5625E-03 & 1.8957 & 1.7846E-03 & 1.9053 \\
	          & 64 & 6.5049E-04 & 1.9780 & 4.5254E-04 & 1.9795 \\
	          & 128 & 1.6405E-04 & 1.9874 & 1.1410E-04 & 1.9877 \\
	          & 256 & 4.1687E-05 & 1.9765 & 2.8987E-05 & 1.9768 \\
		\hline
	\end{tabular}
	\label{tabX1}
\end{table}
%% Table X2
%\begin{table}[H]\tabcolsep=2mm
%	\caption{Numerical errors and the space convergence orders for Example A
%		with $\sigma = 1$ and $M = 2048$.}
%	\centering
%	\begin{tabular}{cccccc}
%		\hline
%		$\alpha$ & $N$ & $Err_\infty(h,\tau)$ & $CO_{\infty, h}$ & $Err_2(h,\tau)$ & $CO_{2, h}$ \\
%		\hline
%		1.2 & 8 & 4.1393E-02 & -- & 3.0009E-02 & -- \\
%		& 16 & 1.9990E-02 & 1.0501 & 1.1631E-02 & 1.3674 \\
%		& 32 & 4.9683E-03 & 2.0085 & 3.0351E-03 & 1.9382 \\
%		& 64 & 1.2171E-03 & 2.0293 & 7.5992E-04 & 1.9978 \\
%		& 128 & 2.9809E-04 & 2.0296 & 1.8683E-04 & 2.0241 \\
%		\hline
%		1.5 & 8 & 5.9195E-02 & -- & 4.4583E-02 & -- \\
%		& 16 & 1.5934E-02 & 1.8934 & 1.2195E-02 & 1.8702 \\
%		& 32 & 4.2812E-03 & 1.8960 & 3.1745E-03 & 1.9417 \\
%		& 64 & 1.0986E-03 & 1.9623 & 8.1033E-04 & 1.9699 \\
%		& 128 & 2.7731E-04 & 1.9861 & 2.0461E-04 & 1.9856 \\
%		\hline
%		1.9 & 8 & 5.9492E-02 & -- & 5.4022E-02 & -- \\
%		& 16 & 1.4177E-02 & 2.0691 & 1.3376E-02 & 2.0139 \\
%		& 32 & 3.6148E-03 & 1.9716 & 3.3459E-03 & 1.9992 \\
%		& 64 & 9.0342E-04 & 2.0004 & 8.3533E-04 & 2.0020 \\
%		& 128 & 2.2594E-04 & 1.9995 & 2.0827E-04 & 2.0039 \\
%		\hline
%	\end{tabular}
%	\label{tabX2}
%\end{table}
% Table X2
\begin{table}[H]\tabcolsep=2mm
	\caption{Numerical errors and the space convergence orders for Example A
		with $\sigma = 1/16$ and $M = 1024$.}
	\centering
	\begin{tabular}{cccccc}
		\hline
		$\alpha$ & $N$ & $Err_\infty(h,\tau)$ & $CO_{\infty, h}$ & $Err_2(h,\tau)$ & $CO_{2, h}$ \\
		\hline
		1.2 & 16 & 1.9992E-02 & -- & 1.1611E-02 & -- \\
	          & 32 & 4.9710E-03 & 2.0078 & 3.0365E-03 & 1.9350 \\
	          & 64 & 1.2200E-03 & 2.0267 & 7.6148E-04 & 1.9955 \\
	          & 128 & 3.0127E-04 & 2.0178 & 1.8847E-04 & 2.0145 \\
	          & 256 & 7.6696E-05 & 1.9738 & 4.7859E-05 & 1.9775 \\
		\hline
		1.5 & 16 & 1.5936E-02 & -- & 1.1349E-02 & -- \\
	          & 32 & 4.1936E-03 & 1.9260 & 3.0132E-03 & 1.9132 \\
	          & 64 & 1.0944E-03 & 1.9380 & 7.7984E-04 & 1.9500 \\
	          & 128 & 2.7917E-04 & 1.9709 & 1.9949E-04 & 1.9669 \\
	          & 256 & 7.2450E-05 & 1.9461 & 5.1615E-05 & 1.9505 \\
		\hline
		1.9 & 16 & 1.2638E-02 & -- & 1.0472E-02 & -- \\
	           & 32 & 3.1667E-03 & 1.9967 & 	2.6441E-03 & 1.9857 \\
	           & 64 & 7.9970E-04 & 1.9854 & 	6.6906E-04 & 1.9826 \\
	           & 128 & 2.0366E-04 & 1.9733 & 1.6975E-04 & 1.9787 \\
	           & 256 & 5.3045E-05 & 1.9409 & 4.3785E-05 & 1.9549 \\
		\hline
	\end{tabular}
	\label{tabX2}
\end{table}

The errors and the convergence orders are listed in Tables \ref{tabX1} and \ref{tabX2}. They show that
the convergence order of our scheme \eqref{eq3.2} is $\mathcal{O}(\tau^2 + h^2)$.
Comparing with the work \cite{duo2018novel}, the convergence order of space is higher than $1 - \frac{\alpha}{2}$.

\section*{References}
\addcontentsline{toc}{section}{References}
\bibliography{references}

\end{document}